\documentclass[twoside,11pt,reqno]{amsart}
\usepackage{amsmath,amsthm,amssymb,amstext,amsfonts,amscd}
\usepackage{graphicx}
\usepackage{multirow}

\newtheorem{theorem}{Theorem}

\newtheorem{corollary}{Corollary}
\newtheorem{definition}{Definition}

\newtheorem{remark}{Remark}

\numberwithin{equation}{section}
\input{tcilatex}

\begin{document}
\title[ Some results relating to $(p,q)$-th relative Gol'dberg order.....]{%
\textbf{Some results relating to }$(p,q)$\textbf{-th relative Gol'dberg
order and }$\left( p,q\right) $\textbf{-relative Gol'dberg type of entire
functions of several variables}}
\author[Tanmay Biswas]{Tanmay Biswas}
\address{T. Biswas : Rajbari, Rabindrapalli, R. N. Tagore Road, P.O.-
Krishnagar, Dist-Nadia, PIN-\ 741101, West Bengal, India}
\email{tanmaybiswas\_math@rediffmail.com}

\begin{abstract}
{\small In this paper we introduce the notions of }$(p,q)${\small -th
relative Gol'dberg order and }$(p,q)${\small -th relative Gol'dberg type of
entire functions of several complex variables where }$p,q${\small \ are any
positive integers. Then we study some growth properties of entire functions
of several complex variables on the basis of their }$\left( p,q\right) $%
{\small -th relative Gol'dberg order and }$(p,q)${\small -th relative
Gol'dberg type.}
\end{abstract}

\maketitle

\section{\textbf{Introduction and Definitions.}}

\qquad Let $%
\mathbb{C}
^{n}$ and $R^{n}$ respectively denote the \emph{complex }and\emph{\ real }$n$%
\emph{-space}. Also let us indicate the point $\left( z_{1},z_{2},\cdot
\cdot \cdot ,z_{n}\right) ,$ $\left( m_{1},m_{2},\cdot \cdot \cdot
,m_{n}\right) $ of $%
\mathbb{C}
^{n}$ or $I^{n}$ by their corresponding unsuffixed symbols $z,m$
respectively where $I$ denotes the set of non-negative integers. The \emph{%
modulus} of $z$, denoted by $\left\vert z\right\vert $, is defined as $%
\left\vert z\right\vert =\left( \left\vert z_{1}\right\vert ^{2}+\cdot \cdot
\cdot +\left\vert z_{n}\right\vert ^{2}\right) ^{\frac{1}{2}}$. If the
coordinates of the vector $m$ are non-negative integers, then $z^{m}$ will
denote $z_{1}^{m_{1}}\cdot \cdot \cdot z_{n}^{m_{n}}$ and $\left\Vert
m\right\Vert =m_{1}+\cdot \cdot \cdot +m_{n}$~.

\qquad If $D\subseteq 
\mathbb{C}
^{n}$ ($%
\mathbb{C}
^{n}$ denote the $n$\emph{-dimensional complex space}) be an arbitrary \emph{%
bounded complex }$n$\emph{-circular domain} with center at the origin of
coordinates then for any entire function $f\left( z\right) $ of $n$ complex
variables and $R>0,$ $M_{f,D}\left( R\right) $ may be define as $%
M_{f,D}\left( R\right) =\QDATOP{\sup }{z\in D_{R}}\left\vert f\left(
z\right) \right\vert $ where a point $z\in D_{R}$ if and only if $\frac{z}{R}%
\in D.$ If $f\left( z\right) $ is non-constant, then $M_{f,D}\left( R\right) 
$ is strictly increasing and its inverse $M_{f,D}^{-1}:\left( \left\vert
f\left( 0\right) \right\vert ,\infty \right) \rightarrow \left( 0,\infty
\right) $ exists such that $\underset{R\rightarrow \infty }{\lim }%
M_{f,D}^{-1}\left( R\right) =\infty .$

\qquad Considering this, the \emph{Gol'dberg order }(resp. \emph{Gol'dberg
lower order }) $\left\{ \text{cf. \cite{4}, \cite{7}}\right\} $ of an entire
function $f\left( z\right) $ with respect to any \emph{bounded complete }$n$%
\emph{-circular domain} $D$ is given by%
\begin{equation*}
\rho _{f,D}=\underset{R\rightarrow +\infty }{\overline{\lim }}\frac{\log ^{%
\left[ 2\right] }M_{f,D}\left( R\right) }{\log R}\text{(resp. }\lambda
_{f,D}=\underset{R\rightarrow +\infty }{\underline{\lim }}\frac{\log ^{\left[
2\right] }M_{f,D}\left( R\right) }{\log R}\text{)}~.
\end{equation*}%
where $\log ^{[k]}R=\log \left( \log ^{[k-1]}R\right) $ for $k=1,2,3,\cdot
\cdot \cdot ;$ $\log ^{[0]}R=R$ and $\exp ^{[k]}R=\exp \left( \exp
^{[k-1]}R\right) $ for $k=1,2,3,\cdot \cdot \cdot $ ; $\exp ^{[0]}R=R.$

\qquad It is well known that $\rho _{f,D}$ is independent of the choice of
the domain $D$, and therefore we write $\rho _{f}$ instead of $\rho _{f,D}$
(resp. $\lambda _{f}$ instead of $\lambda _{f,D}$)$\left\{ \text{cf. \cite{4}%
, \cite{7}}\right\} .$

\qquad For any bounded complete $n$-circular domain $D,$ an entire function
of $n$-complex variables for which\emph{\ Gol'dberg order }and\emph{\
Gol'dberg lower order} are the same is said to be of\emph{\ regular growth}.
Functions which are not of\emph{\ regular growth} are said to be of\emph{\
irregular growth}.

\qquad To compare the relative growth of two entire functions of $n$-complex
variables having same non zero finite \emph{Gol'dberg order}, one may
introduce the definition of \emph{Gol'dberg type }and\emph{\ Gol'dberg lower
type} in the following manner:

\begin{definition}
\label{d1.}$\left\{ \text{cf. \cite{4}, \cite{7}}\right\} $ The \emph{%
Gol'dberg type }and\emph{\ Gol'dberg lower type} respectively denoted by $%
\Delta _{f,D}$ and $\overline{\Delta }_{f,D}$ of an entire function $f\left(
z\right) $ of $n$-complex variables with respect to any \emph{bounded
complete }$n$\emph{-circular domain} $D$ are defined as follows:%
\begin{equation*}
\Delta _{f,D}=\underset{R\rightarrow +\infty }{\overline{\lim }}\frac{\log
M_{f,D}\left( R\right) }{\left( R\right) ^{\rho _{f}}}\ \ \ \text{and \ \ }%
\overline{\Delta }_{f,D}=\underset{R\rightarrow +\infty }{\underline{\lim }}%
\frac{\log M_{f,D}\left( R\right) }{\left( R\right) ^{\rho _{f}}},\text{\ }0<%
\text{ }\rho _{f}\text{ }<+\infty ~.
\end{equation*}
\end{definition}

\qquad Analogously to determine the relative growth of two entire functions
of $n$-complex variables having same non zero finite\emph{\ Gol'dberg lower
order,} one may introduce the definition of\emph{\ Gol'dberg weak type} in
the following way:

\begin{definition}
\label{d2.} The \emph{Gol'dberg weak type} denoted by $\tau _{f,D}$ of an
entire function $f\left( z\right) $ of $n$-complex variables with respect to
any \emph{bounded complete }$n$\emph{-circular domain} $D$ is defined as
follows:%
\begin{equation*}
\tau _{f,D}=\underset{R\rightarrow +\infty }{\underline{\lim }}\frac{\log
M_{f,D}\left( R\right) }{\left( R\right) ^{\lambda _{f}}}~,~0<\text{ }%
\lambda _{f}\text{ }<+\infty ~.
\end{equation*}

Also one may define the growth indicator $\overline{\tau }_{f,D}$ in the
following manner :%
\begin{equation*}
\overline{\tau }_{f,D}=\underset{R\rightarrow +\infty }{\overline{\lim }}%
\frac{\log M_{f,D}\left( R\right) }{\left( R\right) ^{\lambda _{f}}}~,~0<%
\text{ }\lambda _{f}\text{ }<+\infty
\end{equation*}
\end{definition}

\qquad Gol'dberg has shown that \cite{7} \emph{Gol'dberg type }depends on
the domain $D.$ Hence all the growth indicators define in Definition \ref%
{d1.} and Definition \ref{d2.} are also depend on $D.$

\qquad However, extending the notion of \emph{Gol'dberg order}, Datta and
Maji \cite{1} defined the concept of $(p,q)$\emph{-th Gol'dberg order }(resp.%
\emph{\ }$(p,q)$\emph{-th Gol'dberg lower order ) }of an entire function $%
f\left( z\right) $ for any \emph{bounded complete }$n$\emph{-circular domain 
}$D$ where $p\geq q\geq 1$ in the following way:

\begin{eqnarray*}
\rho _{f,D}\left( p,q\right) &=&\text{ }\underset{R\rightarrow +\infty }{%
\overline{\lim }}\frac{\log ^{[p]}M_{f,D}\left( R\right) }{\log ^{\left[ q%
\right] }R}=\underset{R\rightarrow +\infty }{\overline{\lim }}\frac{\log
^{[p]}R}{\log ^{\left[ q\right] }M_{f,D}^{-1}\left( R\right) }\text{ } \\
\text{(resp. }\lambda _{f,D}\left( p,q\right) &=&\underset{R\rightarrow
+\infty }{\underline{\lim }}\frac{\log ^{[p]}M_{f,D}\left( R\right) }{\log ^{%
\left[ q\right] }R}=\underset{R\rightarrow +\infty }{\underline{\lim }}\frac{%
\log ^{[p]}R}{\log ^{\left[ q\right] }M_{f,D}^{-1}\left( R\right) }\text{)}~.
\end{eqnarray*}

\qquad These definitions extended the \emph{generalized Gol'dberg order }$%
\rho _{f,D}^{\left[ l\right] }$\emph{\ }(resp.\emph{\ generalized Gol'dberg
lower order} $\lambda _{f,D}^{\left[ l\right] }$ ) of an entire function $%
f\left( z\right) $ for any \emph{bounded complete }$n$\emph{-circular domain 
}$D$ for each integer $l\geq 2$ since these correspond to the particular
case $\rho _{f,D^{{}}}^{\left[ l\right] }=\rho _{f,D^{{}}}\left( l,1\right) $
(resp. $\lambda _{f,D^{{}}}^{\left[ l\right] }=\lambda _{f,D^{{}}}\left(
l,1\right) $)$.$ Clearly $\rho _{f,D^{{}}}\left( 2,1\right) =\rho
_{f,D^{{}}} $ (resp. $\lambda _{f,D^{{}}}\left( 2,1\right) =\lambda
_{f,D}^{{}}$)$.$ Further in the line of Gol'dberg $\left\{ \text{cf. \cite{4}%
, \cite{7}}\right\} $, one can easily verify that $\rho _{f,D}\left(
p,q\right) $ (resp. $\lambda _{f,D}\left( p,q\right) $) is independent of
the choice of the domain $D$, and therefore one can write $\rho _{f}\left(
p,q\right) $ (resp. $\lambda _{f}\left( p,q\right) $) instead of $\rho
_{f,D}\left( p,q\right) $ (resp. $\lambda _{f,D}\left( p,q\right) $).

\qquad In this connection let us recall that if $0<\rho _{f}\left(
p,q\right) <\infty ,$ then the following properties hold 
\begin{equation*}
\rho _{f}\left( p-n,q\right) =\infty ~\text{for }n<p,~\rho _{f}\left(
p,q-n\right) =0~\text{for }n<q\text{, and}
\end{equation*}%
\begin{equation*}
\rho _{f}\left( p+n,q+n\right) =1~\text{for }n=1,2,...
\end{equation*}

\qquad Similarly for $0<\lambda _{f}\left( p,q\right) <\infty ,$ one can
easily verify that 
\begin{equation*}
\lambda _{f}\left( p-n,q\right) =\infty ~\text{for }n<p,~\lambda _{f}\left(
p,q-n\right) =0~\text{for }n<q\text{, and}
\end{equation*}%
\begin{equation*}
\lambda _{f}\left( p+n,q+n\right) =1~\text{for }n=1,2,....~.
\end{equation*}

\qquad Recalling that for any pair of integer numbers $m,n$ the Kroenecker
function is defined by $\delta _{m,n}=1$ for $m=n$ and $\delta _{m,n}=0$ for 
$m\neq n$, the aforementioned properties provide the following definition.

\begin{definition}
\label{d1} For any bounded complete $n$-circular domain $D,$ an entire
function $f\left( z\right) $ of $n$-complex variables\ is said to have
index-pair $\left( 1,1\right) $ if $0<\rho _{f}\left( 1,1\right) <\infty $.
Otherwise, $f\left( z\right) $ is said to have index-pair $\left( p,q\right)
\neq \left( 1,1\right) $, $p\geq q\geq 1$, if $\delta _{p-q,0}<\rho
_{f}\left( p,q\right) <\infty $ and $\rho _{f}\left( p-1,q-1\right) \notin
R^{+}.$
\end{definition}

\begin{definition}
\label{d2} For any bounded complete $n$-circular domain $D,$ an entire
function $f\left( z\right) $ of $n$-complex variables\ is said to have lower
index-pair $\left( 1,1\right) $ if $0<\lambda _{f}\left( 1,1\right) <\infty $%
. Otherwise, $f\left( z\right) $ is said to have lower index-pair $\left(
p,q\right) \neq \left( 1,1\right) $, $p\geq q\geq 1$, if $\delta
_{p-q,0}<\lambda _{f}\left( p,q\right) <\infty $ and $\lambda _{f}\left(
p-1,q-1\right) \notin R^{+}.$
\end{definition}

\qquad For any \emph{bounded complete }$n$\emph{-circular domain }$D,$ an
entire function $f\left( z\right) $ of $n$\emph{-complex variables\ }of 
\emph{index-pair }$(p,q)$ is said to be of \emph{regular }$\left( p,q\right) 
$\emph{- Gol'dberg growth} if its $\left( p,q\right) $\emph{-th Gol'dberg
order} coincides with its\emph{\ }$(p,q)$\emph{-th Gol'dberg lower order},
otherwise $f\left( z\right) $ is said to be of \emph{irregular }$\left(
p,q\right) $\emph{-Gol'dberg growth}.

\qquad To compare the relative growth of two entire functions having same
non zero finite $\left( p,q\right) $\emph{-Gol'dberg order}, one may
introduce the definition of $\left( p,q\right) $\emph{-Gol'dberg type }and%
\emph{\ }$\left( p,q\right) $-\emph{Gol'dberg lower type} in the following
manner:

\begin{definition}
\label{d5.} The $\left( p,q\right) $\emph{-th Gol'dberg type }and\emph{\ }$%
\left( p,q\right) $-th \emph{Gol'dberg lower type} respectively denoted by $%
\Delta _{f,D}\left( p,q\right) $ and $\overline{\Delta }_{f,D}\left(
p,q\right) $ of an entire function $f\left( z\right) $ of $n$-complex
variables with respect to any \emph{bounded complete }$n$\emph{-circular
domain }$D$ are defined as follows:%
\begin{eqnarray*}
\Delta _{f,D}\left( p,q\right) &=&\underset{R\rightarrow +\infty }{\overline{%
\lim }}\frac{\log ^{\left[ p-1\right] }M_{f,D}\left( R\right) }{\left[ \log
^{\left[ q-1\right] }R\right] ^{\rho _{f}\left( p,q\right) }}\ \ \  \\
\text{and \ \ }\overline{\Delta }_{f,D}\left( p,q\right) &=&\underset{%
R\rightarrow +\infty }{\underline{\lim }}\frac{\log ^{\left[ p-1\right]
}M_{f,D}\left( R\right) }{\left[ \log ^{\left[ q-1\right] }R\right] ^{\rho
_{f}\left( p,q\right) }},\text{\ }0<\text{ }\rho _{f}\left( p,q\right) \text{
}<+\infty ,
\end{eqnarray*}%
where $p\geq q\geq 1~.$
\end{definition}

\qquad An entire function $f\left( z\right) $ of $n$-complex variables of%
\emph{\ index-pair} $(p,q)$ is said to be of \emph{perfectly regular }$%
\left( p,q\right) $\emph{-Gol'dberg growth} if its $\left( p,q\right) $\emph{%
-th Gol'dberg order} coincides with its $(p,q)$\emph{- th Gol'dberg lower
order }as well as its $\left( p,q\right) $\emph{\- th Gol'dberg type }%
coincides with its\emph{\ }$\left( p,q\right) $\emph{\- th Gol'dberg lower
type}.

\qquad Analogously to determine the relative growth of two entire functions
of $n$-complex variables having same non zero finite\emph{\ }$\left(
p,q\right) $-\emph{Gol'dberg lower order,} one may introduce the definition
of\emph{\ }$\left( p,q\right) $- \emph{Gol'dberg weak type} in the following
way:

\begin{definition}
\label{d6..} The $\left( p,q\right) $- th \emph{Gol'dberg weak type} denoted
by $\tau _{f,D}\left( p,q\right) $ of an entire function $f\left( z\right) $
of $n$-complex variables with respect to any \emph{bounded complete }$n$%
\emph{-circular domain }$D$ is defined as follows:%
\begin{equation*}
\tau _{f,D}\left( p,q\right) =\underset{R\rightarrow +\infty }{\underline{%
\lim }}\frac{\log ^{\left[ p-1\right] }M_{f,D}\left( R\right) }{\left[ \log
^{\left[ q-1\right] }R\right] ^{\lambda _{f}\left( p,q\right) }}~,~0<\text{ }%
\lambda _{f}\left( p,q\right) \text{ }<+\infty ~.
\end{equation*}

Also one may define the growth indicator $\overline{\tau }_{f,D}\left(
p,q\right) $ in the following manner :%
\begin{equation*}
\overline{\tau }_{f.D}\left( p,q\right) =\underset{R\rightarrow +\infty }{%
\overline{\lim }}\frac{\log ^{\left[ p-1\right] }M_{f}\left( R\right) }{%
\left[ \log ^{\left[ q-1\right] }R\right] ^{\lambda _{f}\left( p,q\right) }}%
~,~0<\text{ }\lambda _{f}\left( p,q\right) \text{ }<+\infty ,
\end{equation*}%
where $p\geq q\geq 1~.$
\end{definition}

\qquad Definition \ref{d5.} and Definition \ref{d6..} are extended the \emph{%
generalized Gol'dberg type }$\Delta _{f,D}^{\left[ l\right] }$\emph{\ }(resp.%
\emph{generalized Gol'dberg lower type }$\overline{\Delta }_{f,D}^{\left[ l%
\right] })$ and\emph{\ generalized Gol'dberg weak order} $\tau _{f,D}^{\left[
l\right] }$ of an entire function $f\left( z\right) $ of $n$-complex
variables with respect to any \emph{bounded complete }$n$\emph{-circular
domain }$D$ for each integer $l\geq 2$ since these correspond to the
particular case $\Delta _{f,D^{{}}}^{\left[ l\right] }=\Delta
_{f,D^{{}}}\left( l,1\right) $ (resp.\emph{\ }$\overline{\Delta }%
_{f,D^{{}}}^{\left[ l\right] }=\overline{\Delta }_{f^{{}}}\left( l,1\right)
) $ and $\tau _{f,D^{{}}}^{\left[ l\right] }=\tau _{f^{{}}}\left( l,1\right) 
$ (resp.\emph{\ }$\overline{\tau }_{f,D^{{}}}^{\left[ l\right] }=\overline{%
\tau }_{f^{{}}}\left( l,1\right) ).$ Clearly $\Delta _{f,D^{{}}}\left(
2,1\right) =\Delta _{f,D^{{}}}$ (resp.\emph{\ }$\overline{\Delta }%
_{f,D^{{}}}\left( 2,1\right) =\overline{\Delta }_{f,D^{{}}})$ and $\tau
_{f,D^{{}}}\left( 2,1\right) =\tau _{f,D^{{}}}$ (resp.\emph{\ }$\overline{%
\tau }_{f,D^{{}}}\left( 2,1\right) =\overline{\tau }_{f,D^{{}}}).$

\qquad Since Gol'dberg has shown that \cite{7} \emph{Gol'dberg type }depends
on the domain $D,$ therefore all the growth indicators define in Definition %
\ref{d5.} and Definition \ref{d6..} are also depend on $D.$

\qquad For any two entire functions $f\left( z\right) $ and $g\left(
z\right) $ of $n$-complex variables and for any bounded complete $n$%
-circular domain $D$ with center at all the origin $%
\mathbb{C}
^{n},$\ Mondal and Roy \cite{5} introduced the concept \emph{relative
Gol'dberg order} which is as follows:%
\begin{eqnarray*}
\rho _{g,D}\left( f\right) &=&\inf \left\{ \mu >0:M_{f,D}\left( R\right)
<M_{g,D}\left( R^{\mu }\right) \text{ for all }R>R_{0}\left( \mu \right)
>0\right\} \\
&=&\underset{R\rightarrow +\infty }{\overline{\lim }}\frac{\log
M_{g,D}^{-1}M_{f,D}\left( R\right) }{\log R}~.
\end{eqnarray*}

\qquad In \cite{5}, Mandal and Roy also proved that the \emph{relative
Gol'dberg order} of $f\left( z\right) $ with respect to $g\left( z\right) $
is independent of the choice of the domain $D$. So the \emph{relative
Gol'dberg order }of $f\left( z\right) $ with respect to $g\left( z\right) $
may be denoted as $\rho _{g}\left( f\right) $ instead of $\rho _{g,D}\left(
f\right) .$

\qquad Likewise, one can define the \emph{relative Gol'dberg lower order} $%
\lambda _{g,D}\left( f\right) $ in the following manner:%
\begin{equation*}
\lambda _{g,D}\left( f\right) =\underset{R\rightarrow +\infty }{\underline{%
\lim }}\frac{\log M_{g,D}^{-1}M_{f,D}\left( R\right) }{\log R}~.
\end{equation*}

\qquad In the line of Mandal and Roy $\left\{ \text{cf. \cite{5}}\right\} ,$
one can also verify that $\lambda _{g,D}\left( f\right) $ is independent of
the choice of the domain $D$, and therefore one can write $\lambda
_{g}\left( f\right) $ instead of $\lambda _{g,D}\left( f\right) .$

\qquad In the case of \emph{relative Gol'dberg order,} it therefore seems
reasonable to define suitably the $(p,q)$\emph{-th relative Gol'dberg order}
of entire function of $n$\emph{-complex variables} and for any \emph{bounded
complete }$n$\emph{-circular domain} $D$ with center at the origin in $%
\mathbb{C}
^{n}$. With this in view one may introduce the definition of $\left(
p,q\right) $\emph{-th relative Gol'dberg order} $\rho _{g,D}^{\left(
p,q\right) }\left( f\right) $ of an entire function $f\left( z\right) $ with
respect to another entire function $g\left( z\right) $ where both $f\left(
z\right) $ and $g\left( z\right) $ are of $n$-complex variables and $D$ be
any bounded complete $n$-circular domain with center at the origin in $%
\mathbb{C}
^{n}$, in the light of index-pair. Our next definition avoids the
restriction $p>q$ and gives the more natural particular case of \emph{%
Generalized Gol'dberg order }i.e, $\rho _{g,D}^{\left[ l,1\right] }\left(
f\right) =\rho _{g,D}^{\left[ l\right] }\left( f\right) $.

\begin{definition}
\label{d4} Let $f\left( z\right) $ and $g\left( z\right) $ be any two entire
functions of $n$-complex variables with index-pair $\left( m,q\right) $ and $%
\left( m,p\right) ,$ respectively, where $p,q,m$ are positive integers such
that $m\geq q\geq 1$ and $m\geq p\geq 1$ and $D$ be any bounded complete $n$%
-circular domain with center at the origin in $%
\mathbb{C}
^{n}.$ Then the $(p,q)$-th relative Gol'dberg order of $f\left( z\right) $
with respect to $g\left( z\right) $ is defined as 
\begin{eqnarray*}
\rho _{g,D}^{\left( p,q\right) }\left( f\right) &=&\inf \left\{ 
\begin{array}{c}
\mu >0:M_{f,D}\left( r\right) <M_{g,D}\left( \exp ^{\left[ p\right] }\left(
\mu \log ^{\left[ q\right] }R\right) \right) \text{ } \\ 
\text{for all }R>R_{0}\left( \mu \right) >0%
\end{array}%
\right\} ~\ \ \ \ \  \\
&=&\underset{R\rightarrow +\infty }{\overline{\lim }}\frac{\log ^{\left[ p%
\right] }M_{g,D}^{-1}M_{f,D}\left( R\right) }{\log ^{\left[ q\right] }R}=%
\underset{R\rightarrow +\infty }{\overline{\lim }}\frac{\log ^{\left[ p%
\right] }M_{g,D}^{-1}\left( R\right) }{\log ^{\left[ q\right]
}M_{f,D}^{-1}\left( R\right) }~.
\end{eqnarray*}

\qquad Similarly, the $\left( p,q\right) $-th relative Gol'dberg lower order
of $f\left( z\right) $ with respect to $g\left( z\right) $ is defined by: 
\begin{equation*}
\lambda _{g,D}^{\left( p,q\right) }\left( f\right) =\underset{R\rightarrow
+\infty }{\underline{\lim }}\frac{\log ^{\left[ p\right]
}M_{g.D}^{-1}M_{f,D}\left( R\right) }{\log ^{\left[ q\right] }R}=\underset{%
R\rightarrow +\infty }{\underline{\lim }}\frac{\log ^{\left[ p\right]
}M_{g,D}^{-1}\left( R\right) }{\log ^{\left[ q\right] }M_{f,D}^{-1}\left(
R\right) }~.
\end{equation*}
\end{definition}

\qquad In this connection, one may introduce the definition of \emph{%
relative index-pair} of an entire function\emph{\ }with respect to another
entire function (both of $n$-complex variables\emph{) }which is relevant in
the sequel :

\begin{definition}
\label{d5}\cite{xx} Let $f\left( z\right) $ and $g\left( z\right) $ be any
two entire functions (both of $n$-complex variables\emph{) }with index-pairs 
$\left( m,q\right) $ and $\left( m,p\right) $ respectively where $m\geq
q\geq 1$, $m\geq p\geq 1$ and\emph{\ }$D$ be any bounded complete $n$%
-circular domain. Then the entire function $f\left( z\right) $ is said to
have relative index-pair $\left( p,q\right) $ with respect to another entire
function $g\left( z\right) $, if $b<\rho _{g,D}^{\left( p,q\right) }\left(
f\right) <\infty $ and $\rho _{g,D}^{\left( p-1,q-1\right) }\left( f\right) $
is not a nonzero finite number, where $b=1$ if $p=q=m$ and $b=0$ for
otherwise. Moreover if $0<\rho _{g,D}^{\left( p,q\right) }\left( f\right)
<\infty ,$ then%
\begin{equation*}
\rho _{g,D}^{\left( p-n,q\right) }\left( f\right) =\infty ~\text{for }%
n<p,~\rho _{g,D}^{\left( p,q-n\right) }\left( f\right) =0~\text{for }n<q%
\text{ and}
\end{equation*}%
\begin{equation*}
\rho _{g,D}^{\left( p+n,q+n\right) }\left( f\right) =1~\text{for }%
n=1,2,....~.
\end{equation*}%
Similarly for $0<\lambda _{g,D}^{\left( p,q\right) }\left( f\right) <\infty
, $ one can easily verify that%
\begin{equation*}
\lambda _{g,D}^{\left( p-n,q\right) }\left( f\right) =\infty ~\text{for }%
n<p,~\lambda _{g,D}^{\left( p,q-n\right) }\left( f\right) =0~\text{for }n<q%
\text{ and}
\end{equation*}%
\begin{equation*}
\lambda _{g,D}^{\left( p+n,q+n\right) }\left( f\right) =1~\text{for }%
n=1,2,....~.
\end{equation*}
\end{definition}

\qquad Further an entire function $f\left( z\right) $\emph{\ }for which $%
(p,q)$\emph{-th relative Gol'dberg order }and\emph{\ }$(p,q)$\emph{-th
relative Gol'dberg lower order} with respect to another entire function $%
g\left( z\right) $ are the same is called a function of \emph{regular
relative }$(p,q)$\emph{- Gol'dberg growth} with respect to $g\left( z\right) 
$. Otherwise, $f\left( z\right) $ is said to be\emph{\ irregular relative }$%
\left( p,q\right) $\emph{-Gol'dberg growth}.with respect to $g\left(
z\right) $.

\qquad Next we introduce the definition of $(p,q)$\emph{-th relative
Gol'dberg type }and\emph{\ }$(p,q)$-th \emph{relative Gol'dberg lower type }%
in order to compare the relative growth of two entire functions of $n$%
-complex variables having same non zero finite $(p,q)$\emph{-th relative
Gol'dberg order} with respect to another entire function of $n$-complex
variables.

\begin{definition}
\label{d3} Let $f\left( z\right) $ and $g\left( z\right) $ be any two entire
functions of $n$-complex variables with index-pair $\left( m,q\right) $ and $%
\left( m,p\right) ,$ respectively, where $p,q,m$ are positive integers such
that $m\geq q\geq 1$ and $m\geq p\geq 1$ and $D$ be any bounded complete $n$%
-circular domain with center at the origin in $%
\mathbb{C}
^{n}.$ Then the $(p,q)$\emph{- th relative Gol'dberg type} and $(p,q)$\emph{%
- th relative Gol'dberg lower type }of $f\left( z\right) $ with respect to $%
g\left( z\right) $ are defined as%
\begin{eqnarray*}
\Delta _{g,D}^{(p,q)}\left( f\right) &=&\underset{R\rightarrow +\infty }{%
\overline{\lim }}\frac{\log ^{\left[ p-1\right] }M_{g,D}^{-1}M_{f,D}\left(
R\right) }{\left[ \log ^{\left[ q-1\right] }R\right] ^{\rho _{g}^{\left(
p,q\right) }\left( f\right) }}\text{ and} \\
\overline{\Delta }_{g,D}^{(p,q)}\left( f\right) &=&\underset{R\rightarrow
+\infty }{\underline{\lim }}\frac{\log ^{\left[ p-1\right]
}M_{g,D}^{-1}M_{f,D}\left( R\right) }{\left[ \log ^{\left[ q-1\right] }R%
\right] ^{\rho _{g}^{\left( p,q\right) }\left( f\right) }},\text{ }0<\text{ }%
\rho _{g,D}^{\left( p,q\right) }\left( f\right) \text{ }<+\infty ~.
\end{eqnarray*}
\end{definition}

\qquad Analogously to determine the relative growth of two entire functions
of $n$-complex variables having same non zero finite $p,q)$\emph{-th
relative Gol'dberg lower order }with respect to another entire function of $%
n $-complex variables, one may introduce the definition of $(p,q)$\emph{- th
relative Gol'dberg weak type} in the following way:

\begin{definition}
\label{d4...} Let $f\left( z\right) $ and $g\left( z\right) $ be any two
entire functions of $n$-complex variables with index-pair $\left( m,q\right) 
$ and $\left( m,p\right) ,$ respectively, where $p,q,m$ are positive
integers such that $m\geq q\geq 1$ and $m\geq p\geq 1$ and $D$ be any
bounded complete $n$-circular domain with center at the origin in $%
\mathbb{C}
^{n}.$ Then $(p,q)$\emph{- th relative Gol'dberg weak type} denoted by $\tau
_{g,D}^{(p,q)}\left( f\right) $ of an entire function $f\left( z\right) $
with respect to another entire function $g\left( z\right) $ is defined as
follows:%
\begin{equation*}
\tau _{g,D}^{\left( p,q\right) }\left( f\right) =\underset{R\rightarrow
+\infty }{\underline{\lim }}\frac{\log ^{\left[ p-1\right]
}M_{g,D}^{-1}M_{f,D}\left( R\right) }{\left[ \log ^{\left[ q-1\right] }R%
\right] ^{\lambda _{g}^{\left( p,q\right) }\left( f\right) }}~,~0<\text{ }%
\lambda _{g,D}^{\left( p,q\right) }\left( f\right) \text{ }<+\infty ~.
\end{equation*}

Similarly the growth indicator $\overline{\tau }_{g,D}^{\left( p,q\right)
}\left( f\right) $ of an entire function $f\left( z\right) $ with respect to
another entire function $g\left( z\right) $ both of $n$-complex variables in
the following manner :%
\begin{equation*}
\overline{\tau }_{g,D}^{\left( p,q\right) }\left( f\right) =\underset{%
R\rightarrow +\infty }{\overline{\lim }}\frac{\log ^{\left[ p-1\right]
}M_{g,D}^{-1}M_{f,D}\left( R\right) }{\left[ \log ^{\left[ q-1\right] }R%
\right] ^{\lambda _{g}^{\left( p,q\right) }\left( f\right) }}~,~0<\text{ }%
\lambda _{g,D}^{\left( p,q\right) }\left( f\right) \text{ }<+\infty ~.
\end{equation*}
\end{definition}

\qquad Therefore, for any two entire functions $f\left( z\right) $ and $%
g\left( z\right) $ both of $n$-complex variables, we note that%
\begin{eqnarray*}
\rho _{g,D}^{\left( p,q\right) }\left( f\right) &\neq &\lambda
_{g,D}^{\left( p,q\right) }\left( f\right) ,\Delta _{g,D}^{\left( p,q\right)
}\left( f\right) >0\Rightarrow \overline{\tau }_{g,D}^{\left( p,q\right)
}\left( f\right) =+\infty \text{ and} \\
\rho _{g,D}^{\left( p,q\right) }\left( f\right) &\neq &\lambda
_{g,D}^{\left( p,q\right) }\left( f\right) ,\overline{\Delta }_{g,D}^{\left(
p,q\right) }\left( f\right) >0\Rightarrow \tau _{g,D}^{\left( p,q\right)
}\left( f\right) =+\infty ~.
\end{eqnarray*}

\qquad Since Gol'dberg has shown that \cite{7} \emph{Gol'dberg type }depends
on the domain $D.$ Hence all the growth indicators define in Definition \ref%
{d3} and Definition \ref{d4...} are also depend on $D.$

If $f\left( z\right) $ and $g\left( z\right) $ both of $n$-complex variables%
\emph{\ }have got index-pair $\left( m,1\right) $ and $\left( m,l\right) ,$
respectively, then the above two definitions reduces to the definition of 
\emph{generalized relative Gol'dberg type} $\Delta _{g,D}^{\left[ l\right]
}\left( f\right) $ (resp \emph{generalized relative Gol'dberg lower type} $%
\overline{\Delta }_{g,D}^{\left[ l\right] }\left( f\right) $) and \emph{%
generalized relative Gol'dberg weak type} $\tau _{g,D}^{\left[ l\right]
}\left( f\right) $. If the entire functions $f\left( z\right) $ and $g\left(
z\right) $ (both of $n$-complex variables\emph{) }have the same index-pair $%
\left( p,1\right) $ where $p$ is any positive integer, we get the
definitions of \emph{relative Gol'dberg type }as introduced by Roy \cite{8}
( resp.\emph{relative Gol'dberg lower type}) and \emph{relative Gol'dberg
weak type}.

\qquad During the past decades, several authors $\left\{ \text{cf. \cite{1},%
\cite{2},\cite{3},\cite{5},\cite{8}, \cite{6}}\right\} $ made closed
investigations on the properties of entire functions of several complex
variables using different growth indicator such as \emph{Gol'dberg order, }$%
(p,q)$\emph{-th Gol'dberg order }etc. In this paper we wish to measure some
properties of entire functions relative to another entire function of
several complex variables and $D$ will represent a bounded complete $n$%
-circular domain. Actualluy in this paper we wish to study some relative
growth properties of entire functions of $n$-complex variables (all the
entire functions under consideration will be transcendental unless otherwise
stated) using $(p,q)$\emph{-th relative Gol'dberg order, }$(p,q)$- th \emph{%
\emph{relative} Gol'dberg\ type }and\emph{\ }$(p,q)$- th \emph{relative
Gol'dberg weak type.}

\section{\textbf{Main Results.}}

{\normalsize \qquad In this section we state the main results of the paper.}

\begin{theorem}
\label{t1Q} Let $f\left( z\right) $ and $g\left( z\right) $ be any two
entire functions of $n$- complex variables and $D$ be a bounded complete $n$%
-circular domain with center at origin in $%
\mathbb{C}
^{n}.$ Then $(p,q)$-th relative Gol'dberg order $\rho _{g,D}^{\left(
p,q\right) }\left( f\right) $ and $\left( p,q\right) $-th relative Gol'dberg
lower order $\lambda _{g,D}^{\left( p,q\right) }\left( f\right) $ of $%
f\left( z\right) $ with respect to $g\left( z\right) $ is independent of the
choice of the domain $D$ where $p$ and $q$ are any positive integers$.$
\end{theorem}

\begin{proof}
Let us consider $D_{1}$ and $D_{2}$ ne any two bounded complete $n$-circular
domains. Then there exist two real numbers $\alpha $, $\beta >0$ such that $%
\alpha D_{1}\subset D_{2}\subset \beta D_{1}$ and therefore%
\begin{equation*}
M_{f,\alpha D_{1}}\left( R\right) \leq M_{f,D_{2}}\left( R\right) \leq
M_{f,\beta D_{1}}\left( R\right) ~.
\end{equation*}

\qquad Hence for any bounded complete $n$-circular domain $D,$ 
\begin{equation}
M_{g,D}^{-1}\left( M_{f,\alpha D_{1}}\left( R\right) \right) \leq
M_{g,D}^{-1}\left( M_{f,D_{2}}\left( R\right) \right) \leq
M_{g,D}^{-1}\left( M_{f,\beta D_{1}}\left( R\right) \right) ~.  \label{xQ}
\end{equation}

\qquad Now for any $\theta >0$ and any $D,$ we get that%
\begin{equation*}
M_{f,\theta D}\left( R\right) =M_{f,D}\left( \theta R\right) ~.
\end{equation*}

\qquad Therefore 
\begin{eqnarray*}
\underset{R\rightarrow \infty }{\overline{\lim }}\frac{\log ^{\left[ p\right]
}M_{g,D}^{-1}M_{f,\theta D}\left( R\right) }{\log ^{\left[ q\right] }R} &=&%
\underset{R\rightarrow \infty }{\overline{\lim }}\frac{\log ^{\left[ p\right]
}M_{g,D}^{-1}M_{f,D}\left( \theta R\right) }{\log ^{\left[ q\right] }R} \\
&=&\underset{\frac{R}{\theta }\rightarrow \infty }{\overline{\lim }}\frac{%
\log ^{\left[ p\right] }M_{g,D}^{-1}M_{f,D}\left( R\right) }{\log ^{\left[ q%
\right] }\frac{R}{\theta }} \\
&=&\underset{\frac{R}{\theta }\rightarrow \infty }{\overline{\lim }}\frac{%
\log ^{\left[ p\right] }M_{g,D}^{-1}M_{f,D}\left( R\right) }{\log ^{\left[ q%
\right] }R+O(1)} \\
&=&\underset{\frac{R}{\theta }\rightarrow \infty }{\overline{\lim }}\frac{%
\log ^{\left[ p\right] }M_{g,D}^{-1}M_{f,D}\left( R\right) }{\log ^{\left[ q%
\right] }R}~.
\end{eqnarray*}

\qquad Hence by $\left( \ref{xQ}\right) $ we obtain that%
\begin{equation*}
\underset{R\rightarrow \infty }{\overline{\lim }}\frac{\log ^{\left[ p\right]
}M_{g,D_{1}}^{-1}M_{f,D_{1}}\left( R\right) }{\log ^{\left[ q\right] }R}=%
\underset{R\rightarrow \infty }{\overline{\lim }}\frac{\log ^{\left[ p\right]
}M_{g,D_{2}}^{-1}M_{f,D_{2}}\left( R\right) }{\log ^{\left[ q\right] }R}~.
\end{equation*}

\qquad Similarly one can easily verify that%
\begin{equation*}
\underset{R\rightarrow \infty }{\underline{\lim }}\frac{\log ^{\left[ p%
\right] }M_{g,D_{1}}^{-1}M_{f,D_{1}}\left( R\right) }{\log ^{\left[ q\right]
}R}=\underset{R\rightarrow \infty }{\underline{\lim }}\frac{\log ^{\left[ p%
\right] }M_{g,D_{2}}^{-1}M_{f,D_{2}}\left( R\right) }{\log ^{\left[ q\right]
}R}~.
\end{equation*}

\qquad Hence the theorem follows.
\end{proof}

\qquad Since $\rho _{g,D}^{\left( p,q\right) }\left( f\right) $ and $\lambda
_{g,D}^{\left( p,q\right) }\left( f\right) $ are independent of the choice
of the domain $D$, and therefore we write $\rho _{g}^{\left( p,q\right)
}\left( f\right) $ and $\lambda _{g}^{\left( p,q\right) }\left( f\right) $
instead of $\rho _{g,D}^{\left( p,q\right) }\left( f\right) $ and $\lambda
_{g,D}^{\left( p,q\right) }\left( f\right) $ respectively.

\begin{theorem}
\label{l1} Let $f\left( z\right) $ and $g\left( z\right) $ be any two entire
functions of $n$- complex variables with index-pair $\left( m,q\right) $ and 
$\left( m,p\right) ,$ respectively, where $m\geq q\geq 1$ and $m\geq p\geq 1$
and $D$ be a bounded complete $n$-circular domain with center at origin in $%
\mathbb{C}
^{n}.$ Then%
\begin{eqnarray*}
\frac{\lambda _{f}\left( m,q\right) }{\rho _{g}\left( m,p\right) } &\leq
&\lambda _{g}^{\left( p,q\right) }\left( f\right) \leq \min \left\{ \frac{%
\lambda _{f}\left( m,q\right) }{\lambda _{g}\left( m,p\right) },\frac{\rho
_{f}\left( m,q\right) }{\rho _{g}\left( m,p\right) }\right\} \\
&\leq &\max \left\{ \frac{\lambda _{f}\left( m,q\right) }{\lambda _{g}\left(
m,p\right) },\frac{\rho _{f}\left( m,q\right) }{\rho _{g}\left( m,p\right) }%
\right\} \leq \rho _{g}^{\left( p,q\right) }\left( f\right) \leq \frac{\rho
_{f}\left( m,q\right) }{\lambda _{g}\left( m,p\right) }.
\end{eqnarray*}
\end{theorem}

\begin{proof}
From the definitions of $\rho _{g}^{\left( p,q\right) }\left( f\right) $ and 
$\lambda _{g}^{\left( p,q\right) }\left( f\right) $ we get that%
\begin{equation}
\log \rho _{g}^{\left( p,q\right) }\left( f\right) =\underset{R\rightarrow
+\infty }{\overline{\lim }}\left[ \log ^{\left[ p+1\right]
}M_{g,D}^{-1}\left( R\right) -\log ^{\left[ q+1\right] }M_{f,D}^{-1}\left(
R\right) \right] ,  \label{2Q}
\end{equation}%
and%
\begin{equation}
\log \lambda _{g,D}^{\left( p,q\right) }\left( f\right) =\underset{%
R\rightarrow +\infty }{\underline{\lim }}\left[ \log ^{\left[ p+1\right]
}M_{g,D}^{-1}\left( R\right) -\log ^{\left[ q+1\right] }M_{f,D}^{-1}\left(
R\right) \right] ~.  \label{3Q}
\end{equation}

\qquad Now from the definitions of $\rho _{f}\left( m,q\right) $ and $%
\lambda _{f}\left( m,q\right) ,$ it follows that%
\begin{eqnarray}
\log \rho _{f}\left( m,q\right) &=&\underset{R\rightarrow +\infty }{%
\overline{\lim }}\left[ \log ^{[m+1]}R-\log ^{\left[ q+1\right]
}M_{f,D}^{-1}\left( R\right) \right] \text{ and}  \label{4Q} \\
\log \lambda _{f}\left( m,q\right) &=&\underset{R\rightarrow +\infty }{%
\underline{\lim }}\left[ \log ^{[m+1]}R-\log ^{\left[ q+1\right]
}M_{f,D}^{-1}\left( R\right) \right] ~.  \label{5Q}
\end{eqnarray}

\qquad Similarly, from the definitions of $\rho _{g}\left( m,p\right) $ and $%
\lambda _{g}\left( m,p\right) ,$ we obtain that%
\begin{eqnarray}
\log \rho _{g}\left( m,p\right) &=&\underset{R\rightarrow +\infty }{%
\overline{\lim }}\left[ \log ^{[m+1]}R-\log ^{\left[ p+1\right]
}M_{g,D}^{-1}\left( R\right) \right] \text{ and}  \label{6Q} \\
\log \lambda _{g}\left( m,p\right) &=&\underset{R\rightarrow +\infty }{%
\underline{\lim }}\left[ \log ^{[m+1]}R-\log ^{\left[ p+1\right]
}M_{g,D}^{-1}\left( R\right) \right] ~.  \label{7Q}
\end{eqnarray}

\qquad Therefore from $\left( \ref{3Q}\right) ,$ $\left( \ref{5Q}\right) $
and $\left( \ref{6Q}\right) $, we get that%
\begin{equation*}
\log \lambda _{g}^{\left( p,q\right) }\left( f\right) =\underset{%
R\rightarrow +\infty }{\underline{\lim }}\left[ \log ^{[m+1]}R-\log ^{\left[
q+1\right] }M_{f,D}^{-1}\left( R\right) -\left( \log ^{[m+1]}R-\log ^{\left[
p+1\right] }M_{g,D}^{-1}\left( R\right) \right) \right]
\end{equation*}%
\begin{multline*}
i.e.,~\log \lambda _{g}^{\left( p,q\right) }\left( f\right) \geq \left[ 
\underset{R\rightarrow +\infty }{\underline{\lim }}\left( \log
^{[m+1]}R-\log ^{\left[ q+1\right] }M_{f,D}^{-1}\left( R\right) \right)
\right. \\
\left. -\underset{R\rightarrow +\infty }{\overline{\lim }}\left( \log
^{[m+1]}R-\log ^{\left[ p+1\right] }M_{g,D}^{-1}\left( R\right) \right) %
\right]
\end{multline*}%
\begin{equation}
i.e.,~\log \lambda _{g}^{\left( p,q\right) }\left( f\right) \geq \left( \log
\lambda _{f}\left( m,q\right) -\log \rho _{g}\left( m,p\right) \right) ~.
\label{8Q}
\end{equation}

\qquad Similarly, from $\left( \ref{2Q}\right) ,$ $\left( \ref{4Q}\right) $
and $\left( \ref{7Q}\right) $, it follows that%
\begin{equation*}
\log \rho _{g}^{\left( p,q\right) }\left( f\right) =\underset{R\rightarrow
+\infty }{\overline{\lim }}\left[ \log ^{[m+1]}R-\log ^{\left[ q+1\right]
}M_{f,D}^{-1}\left( R\right) -\left( \log ^{[m+1]}R-\log ^{\left[ p+1\right]
}M_{g,D}^{-1}\left( R\right) \right) \right]
\end{equation*}%
\begin{multline*}
i.e.,~\log \rho _{g}^{\left( p,q\right) }\left( f\right) \leq \left[ 
\underset{R\rightarrow +\infty }{\overline{\lim }}\left( \log ^{[m+1]}R-\log
^{\left[ q+1\right] }M_{f,D}^{-1}\left( R\right) \right) \right. \\
\left. -\underset{R\rightarrow +\infty }{\underline{\lim }}\left( \log
^{[m+1]}R-\log ^{\left[ p+1\right] }M_{g,D}^{-1}\left( R\right) \right) %
\right]
\end{multline*}%
\begin{equation}
i.e.,~\log \rho _{g}^{\left( p,q\right) }\left( f\right) \leq \left( \log
\rho _{f}\left( m,q\right) -\log \lambda _{g}\left( m,p\right) \right) ~.
\label{9Q}
\end{equation}

\qquad Again, in view of $\left( \ref{3Q}\right) ,$ $\left( \ref{4Q}\right)
, $ $\left( \ref{5Q}\right) ,$ $\left( \ref{6Q}\right) $ and $\left( \ref{7Q}%
\right) $, we obtain that%
\begin{equation*}
\log \lambda _{g}^{\left( p,q\right) }\left( f\right) =\underset{%
R\rightarrow +\infty }{\underline{\lim }}\left[ \log ^{[m+1]}R-\log ^{\left[
q+1\right] }M_{f,D}^{-1}\left( R\right) -\left( \log ^{[m+1]}R-\log ^{\left[
p+1\right] }M_{g,D}^{-1}\left( R\right) \right) \right]
\end{equation*}%
\begin{equation*}
i.e.,~\log \lambda _{g}^{\left( p,q\right) }\left( f\right) \leq ~\ \ \ \ \
\ \ \ \ \ \ \ \ \ \ \ \ \ \ \ \ \ \ \ \ \ \ \ \ \ \ \ \ \ \ \ \ \ \ \ \ \ \
\ \ \ \ \ \ \ \ \ \ \ \ \ \ \ \ \ \ \ \ \ \ \ \ \ \ \ \ \ \ \ \ \ \ \ \ \ \
\ \ \ \ \ \ \ \ \ \ \ \ \ \ \ \ \ \ \ \ \ \ \ \ \ \ \ \ \ \ \ \ \ \ \ \ \ \
\ \ \ \ \ \ \ \ \ \ \ \ \ \ \ \ \ \ \ \ \ \ \ \ \ \ \ \ \ \ \ \ \ \ \ \ \ \
\ \ \ \ \ \ \ \ \ \ \ \ \ \ \ \ \ \ \ \ \ \ \ \ \ \ \ \ \ \ \ \ \ \ \ \ \ \
\ \ \ \ \ \ \ \ \ \ \ \ \ \ \ \ \ \ \ \ \ \ \ \ \ \ \ \ \ \ \ \ \ \ \ \ \ \
\ \ \ \ \ \ \ \ \ \ \ \ \ \ \ \ \ \ \ 
\end{equation*}%
\begin{multline*}
\min \left[ \underset{R\rightarrow +\infty }{\underline{\lim }}\left( \log
^{[m+1]}R-\log ^{\left[ q+1\right] }M_{f,D}^{-1}\left( R\right) \right) +%
\underset{R\rightarrow +\infty }{\overline{\lim }}-\left( \log
^{[m+1]}R-\log ^{\left[ p+1\right] }M_{g,D}^{-1}\left( R\right) \right)
,\right. \\
\left. \underset{R\rightarrow +\infty }{\overline{\lim }}\left( \log
^{[m+1]}R-\log ^{\left[ q+1\right] }M_{f,D}^{-1}\left( R\right) \right) +%
\underset{R\rightarrow +\infty }{\underline{\lim }}-\left( \log
^{[m+1]}R-\log ^{\left[ p+1\right] }M_{g,D}^{-1}\left( R\right) \right) %
\right]
\end{multline*}%
\begin{equation*}
i.e.,~\log \lambda _{g}^{\left( p,q\right) }\left( f\right) \leq ~\ \ \ \ \
\ \ \ \ \ \ \ \ \ \ \ \ \ \ \ \ \ \ \ \ \ \ \ \ \ \ \ \ \ \ \ \ \ \ \ \ \ \
\ \ \ \ \ \ \ \ \ \ \ \ \ \ \ \ \ \ \ \ \ \ \ \ \ \ \ \ \ \ \ \ \ \ \ \ \ \
\ \ \ \ \ \ \ \ \ \ \ \ \ \ \ \ \ \ \ \ \ \ \ \ \ \ \ \ \ \ \ \ \ \ \ \ \ \
\ \ \ \ \ \ \ \ \ \ \ \ \ \ \ \ \ \ \ \ \ \ \ \ \ \ \ \ \ \ \ \ \ \ \ \ \ \
\ \ \ \ \ \ \ \ \ \ \ \ \ \ \ \ \ \ \ \ \ \ \ \ \ \ \ \ \ \ \ \ \ \ \ \ \ \
\ \ \ \ \ \ \ \ \ \ \ \ \ \ \ \ \ \ \ \ \ \ \ \ \ \ \ \ \ \ \ \ \ \ \ \ \ \
\ \ \ \ \ \ \ \ \ \ \ \ \ \ \ \ \ \ \ 
\end{equation*}%
\begin{multline*}
\min \left[ \underset{R\rightarrow +\infty }{\underline{\lim }}\left( \log
^{[m+1]}R-\log ^{\left[ q+1\right] }M_{f,D}^{-1}\left( R\right) \right) -%
\underset{R\rightarrow +\infty }{\underline{\lim }}\left( \log
^{[m+1]}R-\log ^{\left[ p+1\right] }M_{g,D}^{-1}\left( R\right) \right)
,\right. \\
\left. \underset{R\rightarrow +\infty }{\overline{\lim }}\left( \log
^{[m+1]}R-\log ^{\left[ q+1\right] }M_{f,D}^{-1}\left( R\right) \right) -%
\underset{R\rightarrow +\infty }{\overline{\lim }}\left( \log ^{[m+1]}R-\log
^{\left[ p+1\right] }M_{g,D}^{-1}\left( R\right) \right) \right]
\end{multline*}%
\begin{equation}
i.e.,~\log \lambda _{g}^{\left( p,q\right) }\left( f\right) \leq \min
\left\{ \log \lambda _{f}\left( m,q\right) -\log \lambda _{g}\left(
m,p\right) ,\log \rho _{f}\left( m,q\right) -\log \rho _{g}\left( m,p\right)
\right\} ~.  \label{10Q}
\end{equation}

\qquad Further from $\left( \ref{2Q}\right) ,$ $\left( \ref{4Q}\right) ,$ $%
\left( \ref{5Q}\right) ,$ $\left( \ref{6Q}\right) $ and $\left( \ref{7Q}%
\right) $, it follows that%
\begin{equation*}
\log \rho _{g}^{\left( p,q\right) }\left( f\right) =\underset{R\rightarrow
+\infty }{\overline{\lim }}\left[ \log ^{[m+1]}R-\log ^{\left[ q+1\right]
}M_{f,D}^{-1}\left( R\right) -\left( \log ^{[m+1]}R-\log ^{\left[ p+1\right]
}M_{g,D}^{-1}\left( R\right) \right) \right]
\end{equation*}%
\begin{equation*}
i.e.,~\log \rho _{g}^{\left( p,q\right) }\left( f\right) \geq ~\ \ \ \ \ \ \
\ \ \ \ \ \ \ \ \ \ \ \ \ \ \ \ \ \ \ \ \ \ \ \ \ \ \ \ \ \ \ \ \ \ \ \ \ \
\ \ \ \ \ \ \ \ \ \ \ \ \ \ \ \ \ \ \ \ \ \ \ \ \ \ \ \ \ \ \ \ \ \ \ \ \ \
\ \ \ \ \ \ \ \ \ \ \ \ \ \ \ \ \ \ \ \ \ \ \ \ \ \ \ \ \ \ \ \ \ \ \ \ \ \
\ \ \ \ \ \ \ \ \ \ \ \ \ \ \ \ \ \ \ \ \ \ \ \ \ \ \ \ \ \ \ \ \ \ \ \ \ \
\ \ \ \ \ \ \ \ \ \ \ \ \ \ \ \ \ \ \ \ \ \ \ \ \ \ \ \ \ \ \ \ \ \ \ \ \ \
\ \ \ \ \ \ \ \ \ \ \ \ \ \ \ \ \ \ \ \ \ \ \ \ \ \ \ \ \ \ \ \ \ \ \ \ \ \
\ \ \ \ \ \ \ \ \ \ \ \ \ \ \ \ \ 
\end{equation*}%
\begin{multline*}
\max \left[ \underset{R\rightarrow +\infty }{\underline{\lim }}\left( \log
^{[m+1]}R-\log ^{\left[ q+1\right] }M_{f,D}^{-1}\left( R\right) \right) +%
\underset{R\rightarrow +\infty }{\overline{\lim }}-\left( \log
^{[m+1]}R-\log ^{\left[ p+1\right] }M_{g,D}^{-1}\left( R\right) \right)
,\right. \\
\left. \underset{R\rightarrow +\infty }{\overline{\lim }}\left( \log
^{[m+1]}R-\log ^{\left[ q+1\right] }M_{f,D}^{-1}\left( R\right) \right) +%
\underset{R\rightarrow +\infty }{\underline{\lim }}-\left( \log
^{[m+1]}R-\log ^{\left[ p+1\right] }M_{g,D}^{-1}\left( R\right) \right) %
\right]
\end{multline*}%
\begin{equation*}
i.e.,~\log \rho _{g}^{\left( p,q\right) }\left( f\right) \geq ~\ \ \ \ \ \ \
\ \ \ \ \ \ \ \ \ \ \ \ \ \ \ \ \ \ \ \ \ \ \ \ \ \ \ \ \ \ \ \ \ \ \ \ \ \
\ \ \ \ \ \ \ \ \ \ \ \ \ \ \ \ \ \ \ \ \ \ \ \ \ \ \ \ \ \ \ \ \ \ \ \ \ \
\ \ \ \ \ \ \ \ \ \ \ \ \ \ \ \ \ \ \ \ \ \ \ \ \ \ \ \ \ \ \ \ \ \ \ \ \ \
\ \ \ \ \ \ \ \ \ \ \ \ \ \ \ \ \ \ \ \ \ \ \ \ \ \ \ \ \ \ \ \ \ \ \ \ \ \
\ \ \ \ \ \ \ \ \ \ \ \ \ \ \ \ \ \ \ \ \ \ \ \ \ \ \ \ \ \ \ \ \ \ \ \ \ \
\ \ \ \ \ \ \ \ \ \ \ \ \ \ \ \ \ \ \ \ \ \ \ \ \ \ \ \ \ \ \ \ \ \ \ \ \ \
\ \ \ \ \ \ \ \ \ \ \ \ \ \ \ \ \ 
\end{equation*}%
\begin{multline*}
\max \left[ \underset{R\rightarrow +\infty }{\underline{\lim }}\left( \log
^{[m+1]}R-\log ^{\left[ q+1\right] }M_{f,D}^{-1}\left( R\right) \right) -%
\underset{R\rightarrow +\infty }{\underline{\lim }}\left( \log
^{[m+1]}R-\log ^{\left[ p+1\right] }M_{g,D}^{-1}\left( R\right) \right)
,\right. \\
\left. \underset{R\rightarrow +\infty }{\overline{\lim }}\left( \log
^{[m+1]}R-\log ^{\left[ q+1\right] }M_{f,D}^{-1}\left( R\right) \right) -%
\underset{R\rightarrow +\infty }{\overline{\lim }}\left( \log ^{[m+1]}R-\log
^{\left[ p+1\right] }M_{g,D}^{-1}\left( R\right) \right) \right]
\end{multline*}%
\begin{equation}
i.e.,~\log \rho _{g}^{\left( p,q\right) }\left( f\right) \geq \max \left\{
\log \lambda _{f}\left( m,q\right) -\log \lambda _{g}\left( m,p\right) ,\log
\rho _{f}\left( m,q\right) -\log \rho _{g}\left( m,p\right) \right\} ~.
\label{11Q}
\end{equation}

\qquad Thus the theorem follows from $\left( \ref{8Q}\right) ,$ $\left( \ref%
{9Q}\right) ,$ $\left( \ref{10Q}\right) $ and $\left( \ref{11Q}\right) ~.$
\end{proof}

\qquad In view of Theorem \ref{l1}, one can easily verify the following
corollaries:

\begin{corollary}
\label{c21Q} Let $f\left( z\right) $ be an entire function of $n$- complex
variables with index-pair $\left( m,q\right) $ and and $g\left( z\right) $
be any two entire functions of $n$- complex variables of regular $\left(
m,p\right) $-th Gol'dberg growth where $m\geq q\geq 1$ and $m\geq p\geq 1$
and $D$ be a bounded complete $n$-circular domain with center at origin in $%
\mathbb{C}
^{n}.$ Then%
\begin{equation*}
\lambda _{g}^{\left( p,q\right) }\left( f\right) =\frac{\lambda _{f}\left(
m,q\right) }{\rho _{g}\left( m,p\right) }\text{ \ \ and \ \ }\rho
_{g}^{\left( p,q\right) }\left( f\right) =\frac{\rho _{f}\left( m,q\right) }{%
\rho _{g}\left( m,p\right) }~.
\end{equation*}%
Moreover, if $\rho _{f}\left( m,q\right) =\rho _{g}\left( m,p\right) ,$ then 
\begin{equation*}
\rho _{g}^{\left( p,q\right) }\left( f\right) =\lambda _{f}^{\left(
q,p\right) }\left( g\right) =1~.
\end{equation*}
\end{corollary}

\begin{corollary}
\label{c31Q} Let $f$ and $g$ be any two entire functions of $n$- complex
variables\ and\ with regular $\left( m,q\right) $-th Gol'dberg growth and
regular $\left( m,p\right) $-th Gol'dberg growth, respectively, where $m\geq
q\geq 1$ and $m\geq p\geq 1$. Also and $D$ be a bounded complete $n$%
-circular domain with center at origin in $%
\mathbb{C}
^{n}.$ Then%
\begin{equation*}
\lambda _{g}^{\left( p,q\right) }\left( f\right) =\rho _{g}^{\left(
p,q\right) }\left( f\right) =\frac{\rho _{f}\left( m,q\right) }{\rho
_{g}\left( m,p\right) }~.
\end{equation*}
\end{corollary}

\begin{corollary}
\label{c41Q} Let $f$ and $g$ be any two entire functions of $n$- complex
variables\ and\ with regular $\left( m,q\right) $-th Gol'dberg growth and
regular $\left( m,p\right) $-th Gol'dberg growth, respectively, where $m\geq
q\geq 1$ and $m\geq p\geq 1$. Also and $D$ be a bounded complete $n$%
-circular domain with center at origin in $%
\mathbb{C}
^{n}$ and $\rho _{f}\left( m,q\right) =\rho _{g}\left( m,p\right) .$ Then%
\begin{equation*}
\lambda _{g}^{\left( p,q\right) }\left( f\right) =\rho _{g}^{\left(
p,q\right) }\left( f\right) =\lambda _{f}^{\left( q,p\right) }\left(
g\right) =\rho _{f}^{\left( q,p\right) }\left( g\right) =1~.
\end{equation*}
\end{corollary}

\begin{corollary}
\label{c51Q} Let $f$ and $g$ be any two entire functions of $n$- complex
variables\ and\ with regular $\left( m,q\right) $-th Gol'dberg growth and
regular $\left( m,p\right) $-th Gol'dberg growth, respectively, where $m\geq
q\geq 1$ and $m\geq p\geq 1$. Also and $D$ be a bounded complete $n$%
-circular domain with center at origin in $%
\mathbb{C}
^{n}$. Then 
\begin{equation*}
\rho _{g}^{\left( p,q\right) }\left( f\right) .\rho _{f}^{\left( q,p\right)
}\left( g\right) =\lambda _{g}^{\left( p,q\right) }\left( f\right) .\lambda
_{f}^{\left( q,p\right) }\left( g\right) =1~.
\end{equation*}
\end{corollary}

\begin{corollary}
\label{c61Q} Let $f\left( z\right) $ and $g\left( z\right) $ be any two
entire functions of $n$- complex variables with index-pair $\left(
m,q\right) $ and $\left( m,p\right) ,$ respectively, where $m\geq q\geq 1$
and $m\geq p\geq 1$ and $D$ be a bounded complete $n$-circular domain with
center at origin in $%
\mathbb{C}
^{n}.$ If either $f$ is not of regular $\left( m,q\right) $-th Gol'dberg
growth or $g$ is not of regular $\left( m,p\right) $-th Gol'dberg growth,
then 
\begin{equation*}
\lambda _{g}^{\left( p,q\right) }\left( f\right) .\lambda _{f}^{\left(
q,p\right) }\left( g\right) <1~<\rho _{g}^{\left( p,q\right) }\left(
f\right) .\rho _{f}^{\left( q,p\right) }\left( g\right) .
\end{equation*}
\end{corollary}

\begin{corollary}
\label{c71Q} Let $f\left( z\right) $ be an entire function of $n$- complex
variables with index-pair $\left( m,q\right) $ where $m\geq q\geq 1$ and $D$
be a bounded complete $n$-circular domain with center at origin in $%
\mathbb{C}
^{n}.$ Then for any entire function $g$ of $n$- complex variables%
\begin{eqnarray*}
\left( i\right) ~\lambda _{g}^{\left( p,q\right) }\left( f\right) &=&\infty ~%
\text{when }\rho _{g}\left( m,p\right) =0, \\
\left( ii\right) ~\rho _{g}^{\left( p,q\right) }\left( f\right) &=&\infty ~%
\text{when }\lambda _{g}\left( m,p\right) =0, \\
\left( iii\right) ~\lambda _{g}^{\left( p,q\right) }\left( f\right) &=&0~%
\text{when }\rho _{g}\left( m,p\right) =\infty
\end{eqnarray*}%
and 
\begin{equation*}
\left( iv\right) ~\rho _{g}^{\left( p,q\right) }\left( f\right) =0~\text{%
when }\lambda _{g}\left( m,p\right) =\infty ,
\end{equation*}%
where $m\geq p\geq 1.$
\end{corollary}

\begin{corollary}
\label{c81Q} Let $g\left( z\right) $ be an entire function of $n$- complex
variables with index-pair $\left( m,p\right) $ where $m\geq p\geq 1$ and $D$
be a bounded complete $n$-circular domain with center at origin in $%
\mathbb{C}
^{n}.$ Then for any entire function $f$ of $n$- complex variables%
\begin{eqnarray*}
\left( i\right) ~\rho _{g}^{\left( p,q\right) }\left( f\right) &=&0~\text{%
when }\rho _{f}\left( m,q\right) =0, \\
\left( ii\right) ~\lambda _{g}^{\left( p,q\right) }\left( f\right) &=&0~%
\text{when }\lambda _{f}\left( m,q\right) =0, \\
\left( iii\right) ~\rho _{g}^{\left( p,q\right) }\left( f\right) &=&\infty ~%
\text{when }\rho _{f}\left( m,q\right) =\infty ,
\end{eqnarray*}%
and 
\begin{equation*}
\left( iv\right) ~\lambda _{g}^{\left( p,q\right) }\left( f\right) =\infty ~%
\text{when }\lambda _{f}\left( m,q\right) =\infty ,
\end{equation*}%
where $m\geq q\geq 1$.
\end{corollary}

\begin{remark}
\label{r1} From the conclusion Theorem \ref{l1}, one may write $\rho
_{g}^{\left( p,q\right) }\left( f\right) =\frac{\rho _{f}\left( m,q\right) }{%
\rho _{g}\left( m,p\right) }$ \ and \ $\lambda _{g}^{\left( p,q\right)
}\left( f\right) =\frac{\lambda _{f}\left( m,q\right) }{\lambda _{g}\left(
m,p\right) }$ when $g\left( z\right) $ be an entire function of $n$- complex
variables with regular $\left( m,p\right) $-Gol'dberg growth. Similarly $%
\rho _{g}^{\left( p,q\right) }\left( f\right) =\frac{\lambda _{f}\left(
m,q\right) }{\lambda _{g}\left( m,p\right) }$\ and \ $\lambda _{g}^{\left(
p,q\right) }\left( f\right) =\frac{\rho _{f}\left( m,q\right) }{\rho
_{g}\left( m,p\right) }$ when $f\left( z\right) $ be an entire function of $%
n $- complex variables with regular $\left( m,q\right) $- Gol'dberg growth.
\end{remark}

\qquad When $f\left( z\right) $ and $g\left( z\right) $ are any two entire
functions both of $n$- complex variables\ and with index-pair $\left(
m,q\right) $ and $\left( n,p\right) ,$ respectively, where $m\geq q+1\geq 1$
and $n\geq p+1\geq 1,$ but $m\neq n$, the next definition enables us to
study their relative order for any bounded complete $n$-circular domain $D$
with center at origin in $%
\mathbb{C}
^{n}.$

\begin{definition}
\label{d6Q} Let $f\left( z\right) $ and $g\left( z\right) $ be any two
entire functions of $n$- complex variables with index-pair $\left(
m,q\right) $ and $\left( n,p\right) ,$ respectively, where $m\geq q\geq 1$
and $n\geq p\geq 1$ and $D$ be a bounded complete $n$-circular domain with
center at origin in $%
\mathbb{C}
^{n}.$ then the $\left( p+m-n,q\right) $-th relative Gol'dberg order (resp. $%
\left( p+m-n,q\right) $-th relative Gol'dberg lower order) of $f\left(
z\right) $ with respect to $g\left( z\right) $ is defined as 
\begin{eqnarray*}
\left( i\right) ~\rho _{g}^{\left( p+m-n,q\right) }\left( f\right) &=&\text{ 
}\underset{R\rightarrow +\infty }{\overline{\lim }}\frac{\log ^{\left[ p+m-n%
\right] }M_{g,D}^{-1}M_{f,D}\left( R\right) }{\log ^{\left[ q\right] }R} \\
(\text{resp. }\lambda _{g}^{\left( p+m-n,q\right) }\left( f\right) &=&%
\underset{R\rightarrow +\infty }{\underline{\lim }}\frac{\log ^{\left[ p+m-n%
\right] }M_{g,D}^{-1}M_{f,D}\left( R\right) }{\log ^{\left[ q\right] }R}).
\end{eqnarray*}%
If $m<n$, then the $\left( p,q+n-m\right) $-th relative Gol'dberg order
(resp. $\left( p,q+n-m\right) $-th relative Gol'dberg lower order) of $%
f\left( z\right) $ with respect to $g\left( z\right) $ is defined as 
\begin{eqnarray*}
\left( ii\right) ~\rho _{g}^{\left( p,q+n-m\right) }\left( f\right) &=&%
\underset{R\rightarrow +\infty }{\overline{\lim }}\frac{\log ^{\left[ p%
\right] }M_{g,D}^{-1}M_{f,D}\left( R\right) }{\log ^{\left[ q+n-m\right] }R}
\\
(\text{resp. }\lambda _{g}^{\left( p,q+n-m\right) }\left( f\right) &=&%
\underset{R\rightarrow +\infty }{\underline{\lim }}\frac{\log ^{\left[ p%
\right] }M_{g,D}^{-1}M_{f,D}\left( R\right) }{\log ^{\left[ q+n-m\right] }R}%
).
\end{eqnarray*}
\end{definition}

\qquad Move to the left.

\begin{theorem}
\label{t3Q}Under the hypothesis of Definition \ref{d6Q}, for $m>n:$ 
\begin{eqnarray*}
\left( i\right) ~\rho _{g}^{\left( p+m-n,q\right) }\left( f\right) &=&%
\underset{R\rightarrow +\infty }{\overline{\lim }}\frac{\log ^{\left[ m%
\right] }M_{f,D}\left( R\right) }{\log ^{\left[ q\right] }R}\text{,} \\
\lambda _{g}^{\left( p+m-n,q\right) }\left( f\right) &=&\underset{%
R\rightarrow +\infty }{\underline{\lim }}\frac{\log ^{\left[ m\right]
}M_{f,D}\left( R\right) }{\log ^{\left[ q\right] }R}.
\end{eqnarray*}%
and for $m<n:$ 
\begin{eqnarray*}
\left( ii\right) ~\rho _{g}^{\left( p,q+n-m\right) }\left( f\right) &=&%
\underset{R\rightarrow +\infty }{\overline{\lim }}\frac{\log ^{\left[ p%
\right] }R}{\log ^{\left[ n\right] }M_{g,D}\left( R\right) }\text{, } \\
\lambda _{g}^{\left( p,q+n-m\right) }\left( f\right) &=&\underset{%
R\rightarrow +\infty }{\underline{\lim }}\frac{\log ^{\left[ p\right] }R}{%
\log ^{\left[ n\right] }M_{g,D}\left( R\right) }\text{ }.
\end{eqnarray*}
\end{theorem}

\qquad In the next theorem we intend to find out $\left( p,q\right) $-th
relative Gol'dberg order ( resp. $\left( p,q\right) $-th relative Gol'dberg
lower order ) of an entire function $f\left( z\right) $\ with respect to
another entire function $g\left( z\right) $\ (both $f\left( z\right) $ and $%
g\left( z\right) $ are of $n$- complex variables ) when $\left( m,q\right) $%
-th relative Gol'dberg order (resp. $\left( m,q\right) $-th relative
Gol'dberg lower order) of $f\left( z\right) $\ and $\left( m,p\right) $-th
relative Gol'dberg order (resp. $\left( m,p\right) $-th relative Gol'dberg
lower order) of $g\left( z\right) $\ with respect to another entire function 
$h\left( z\right) $\ ($h\left( z\right) $ is also of $n$- complex variables
) are given where $p,q$ and $m$ are any positive integers$.$

\begin{theorem}
\label{l2} Let $f\left( z\right) $, $g\left( z\right) $ and $h\left(
z\right) $ be any three entire functions of $n$- complex variables and $D$
be a bounded complete $n$-circular domain with center at origin in $%
\mathbb{C}
^{n}.$ Also let $m,p,q$ are any three positive integers. If $\left(
m,q\right) $-th relative Gol'dberg order (resp. $\left( m,q\right) $-th
relative Gol'dberg lower order) of $f\left( z\right) $ with respect to $%
h\left( z\right) $ and $\left( m,p\right) $-th relative Gol'dberg order
(resp. $\left( m,p\right) $-th relative Gol'dberg lower order) of $g\left(
z\right) $ with respect to $h\left( z\right) $ are respectively denoted by $%
\rho _{h}^{\left( m,q\right) }\left( f\right) $ $\left( \text{resp. }\lambda
_{h}^{\left( m,q\right) }\left( f\right) \right) $ and $\rho _{h}^{\left(
m,p\right) }\left( g\right) $ $\left( \text{resp. }\lambda _{h}^{\left(
m,p\right) }\left( g\right) \right) $, then%
\begin{eqnarray*}
\frac{\lambda _{h}^{\left( m,q\right) }\left( f\right) }{\rho _{h}^{\left(
m,p\right) }\left( g\right) } &\leq &\lambda _{g}^{\left( p,q\right) }\left(
f\right) \leq \min \left\{ \frac{\lambda _{h}^{\left( m,q\right) }\left(
f\right) }{\lambda _{h}^{\left( m,p\right) }\left( g\right) },\frac{\rho
_{h}^{\left( m,q\right) }\left( f\right) }{\rho _{h}^{\left( m,p\right)
}\left( g\right) }\right\} \\
&\leq &\max \left\{ \frac{\lambda _{h}^{\left( m,q\right) }\left( f\right) }{%
\lambda _{h}^{\left( m,p\right) }\left( g\right) },\frac{\rho _{h}^{\left(
m,q\right) }\left( f\right) }{\rho _{h}^{\left( m,p\right) }\left( g\right) }%
\right\} \leq \rho _{g}^{\left( p,q\right) }\left( f\right) \leq \frac{\rho
_{h}^{\left( m,q\right) }\left( f\right) }{\lambda _{h}^{\left( m,p\right)
}\left( g\right) }~.
\end{eqnarray*}
\end{theorem}

\qquad The conclusion of the above theorem can be carried out after applying
the same technique of Theorem \ref{l1} and therefore its proof is omitted.

\qquad In view of Theorem \ref{l2}, one can easily verify the following
corollaries:

\begin{corollary}
\label{c1Q} Let $f\left( z\right) $, $g\left( z\right) $ and $h\left(
z\right) $ be any three entire functions of $n$- complex variables and $D$
be a bounded complete $n$-circular domain with center at origin in $%
\mathbb{C}
^{n}.$ Also let $f\left( z\right) $ be an entire function with regular
relative $(m,q)$-Gol'dberg growth with respect to entire function $h\left(
z\right) $ and $g\left( z\right) $ be entire having relative index-pair $%
\left( m,p\right) $ with respect to another entire function $h\left(
z\right) $ where $m,p,q$ are any three positive integers. Then%
\begin{equation*}
\lambda _{g}^{\left( p,q\right) }\left( f\right) =\frac{\rho _{h}^{\left(
m,q\right) }\left( f\right) }{\rho _{h}^{\left( m,p\right) }\left( g\right) }%
\text{\ \ \ and \ \ }\rho _{g}^{\left( p,q\right) }\left( f\right) =\frac{%
\rho _{h}^{\left( m,q\right) }\left( f\right) }{\lambda _{h}^{\left(
m,p\right) }\left( g\right) }~.
\end{equation*}%
In addition, if $\rho _{h}^{\left( m,q\right) }\left( f\right) =\rho
_{h}^{\left( m,p\right) }\left( g\right) ,$ then%
\begin{equation*}
\lambda _{g}^{\left( p,q\right) }\left( f\right) =\rho _{f}^{\left(
q,p\right) }\left( g\right) =1~.
\end{equation*}
\end{corollary}

\begin{corollary}
\label{c2Q} Let $f\left( z\right) $, $g\left( z\right) $ and $h\left(
z\right) $ be any three entire functions of $n$- complex variables and $D$
be a bounded complete $n$-circular domain with center at origin in $%
\mathbb{C}
^{n}.$ Also let $f\left( z\right) $ be an entire function with relative
index-pair $\left( m,q\right) $ with respect to entire function $h\left(
z\right) $ and $g\left( z\right) $ be entire of regular relative $(m,p)$%
-Gol'dberg growth with respect to another entire function $h\left( z\right) $
where $m,p,q$ are any three positive integers. Then%
\begin{equation*}
\lambda _{g}^{\left( p,q\right) }\left( f\right) =\frac{\lambda _{h}^{\left(
m,q\right) }\left( f\right) }{\rho _{h}^{\left( m,p\right) }\left( g\right) }%
\text{ \ \ and \ \ }\rho _{g}^{\left( p,q\right) }\left( f\right) =\frac{%
\rho _{h}^{\left( m,q\right) }\left( f\right) }{\rho _{h}^{\left( m,p\right)
}\left( g\right) }~.
\end{equation*}%
In addition, if $\rho _{h}^{\left( m,q\right) }\left( f\right) =\rho
_{h}^{\left( m,p\right) }\left( g\right) ,$ then%
\begin{equation*}
\rho _{g}^{\left( p,q\right) }\left( f\right) =\lambda _{f}^{\left(
q,p\right) }\left( g\right) =1~.
\end{equation*}
\end{corollary}

\begin{corollary}
\label{c3Q} Let $f\left( z\right) $, $g\left( z\right) $ and $h\left(
z\right) $ be any three entire functions of $n$- complex variables and $D$
be a bounded complete $n$-circular domain with center at origin in $%
\mathbb{C}
^{n}.$ Also let $f\left( z\right) $ and $g\left( z\right) $ be any two
entire functions with regular relative $(m,q)$-Gol'dberg growth and regular
relative $(m,p)$-th Gol'dberg growth with respect to entire function $%
h\left( z\right) $ respectively where $m,p,q$ are any three positive
integers. Then%
\begin{equation*}
\lambda _{g}^{\left( p,q\right) }\left( f\right) =\rho _{g}^{\left(
p,q\right) }\left( f\right) =\frac{\rho _{h}^{\left( m,q\right) }\left(
f\right) }{\rho _{h}^{\left( m,p\right) }\left( g\right) }~.
\end{equation*}
\end{corollary}

\begin{corollary}
\label{c4Q} Let $f\left( z\right) $, $g\left( z\right) $ and $h\left(
z\right) $ be any three entire functions of $n$- complex variables and $D$
be a bounded complete $n$-circular domain with center at origin in $%
\mathbb{C}
^{n}.$ Also let $f\left( z\right) $ and $g\left( z\right) $ be any two
entire functions with regular relative $(m,q)$- Gol'dberg growth and regular
relative $(m,p)$- Gol'dberg growth with respect to entire function $h\left(
z\right) $ respectively where $m,p,q$ are any three positive integers.Then%
\begin{equation*}
\lambda _{g}^{\left( p,q\right) }\left( f\right) =\rho _{g}^{\left(
p,q\right) }\left( f\right) =\lambda _{f}^{\left( q,p\right) }\left(
g\right) =\rho _{f}^{\left( q,p\right) }\left( g\right) =1
\end{equation*}%
if $\rho _{h}^{\left( m,q\right) }\left( f\right) =\rho _{h}^{\left(
m,p\right) }\left( g\right) $.
\end{corollary}

\begin{corollary}
\label{c5Q} Let $f\left( z\right) $, $g\left( z\right) $ and $h\left(
z\right) $ be any three entire functions of $n$- complex variables and $D$
be a bounded complete $n$-circular domain with center at origin in $%
\mathbb{C}
^{n}.$ Also let $f\left( z\right) $ and $g\left( z\right) $ be any two
entire functions with relative index-pairs $\left( m,q\right) $ and $\left(
m,p\right) $ with respect to entire function $h\left( z\right) $
respectively where $m,p,q$ are any three positive integers and either $%
f\left( z\right) $ is not of regular relative $\left( m,q\right) $ -
Gol'dberg growth or $g\left( z\right) $ is not of regular relative $\left(
m,p\right) $ - Gol'dberg growth, then%
\begin{equation*}
\rho _{g}^{\left( p,q\right) }\left( f\right) .\rho _{f}^{\left( q,p\right)
}\left( g\right) \geq 1~.
\end{equation*}%
If $f\left( z\right) $ and $g\left( z\right) $ are both of regular relative $%
\left( m,q\right) $- Gol'dberg growth and regular relative $\left(
m,p\right) $ - Gol'dberg growth with respect to entire function $h\left(
z\right) $ respectively, then%
\begin{equation*}
\rho _{g}^{\left( p,q\right) }\left( f\right) .\rho _{f}^{\left( q,p\right)
}\left( g\right) =1~.
\end{equation*}
\end{corollary}

\begin{corollary}
\label{c6Q} Let $f\left( z\right) $, $g\left( z\right) $ and $h\left(
z\right) $ be any three entire functions of $n$- complex variables and $D$
be a bounded complete $n$-circular domain with center at origin in $%
\mathbb{C}
^{n}.$ Also let $f\left( z\right) $ and $g\left( z\right) $ be any two
entire functions with relative index-pairs $\left( m,q\right) $ and $\left(
m,p\right) $ with respect to entire function $h\left( z\right) $
respectively where $m,p,q$ are any three positive integers and either $%
f\left( z\right) $ is not of regular relative $\left( m,q\right) $ -
Gol'dberg growth or $g\left( z\right) $ is not of regular relative $\left(
m,p\right) $ - Gol'dberg growth, then%
\begin{equation*}
\lambda _{g}^{\left( p,q\right) }\left( f\right) .\lambda _{f}^{\left(
q,p\right) }\left( g\right) \leq 1~.
\end{equation*}%
If $f\left( z\right) $ and $g\left( z\right) $ are both of regular relative $%
\left( m,q\right) $ - Gol'dberg growth and regular relative $\left(
m,p\right) $ -Gol'dberg growth with respect to entire function $h\left(
z\right) $ respectively, then%
\begin{equation*}
\lambda _{g}^{\left( p,q\right) }\left( f\right) .\lambda _{f}^{\left(
q,p\right) }\left( g\right) =1~.
\end{equation*}
\end{corollary}

\begin{corollary}
\label{c7Q} Let $f\left( z\right) $, $g\left( z\right) $ and $h\left(
z\right) $ be any three entire functions of $n$- complex variables and $D$
be a bounded complete $n$-circular domain with center at origin in $%
\mathbb{C}
^{n}.$ Also let $f\left( z\right) $ be an entire function with relative
index-pair $\left( m,q\right) $, Then%
\begin{eqnarray*}
\left( i\right) ~\lambda _{g}^{\left( p,q\right) }\left( f\right) &=&\infty ~%
\text{when }\rho _{h}^{\left( m,p\right) }\left( g\right) =0~, \\
\left( ii\right) ~\rho _{g}^{\left( p,q\right) }\left( f\right) &=&\infty ~%
\text{when }\lambda _{h}^{\left( m,p\right) }\left( g\right) =0~, \\
\left( iii\right) ~\lambda _{g}^{\left( p,q\right) }\left( f\right) &=&0~%
\text{when }\rho _{h}^{\left( m,p\right) }\left( g\right) =\infty
\end{eqnarray*}%
and%
\begin{equation*}
\left( iv\right) ~\rho _{g}^{\left( p,q\right) }\left( f\right) =0~\text{%
when }\lambda _{h}^{\left( m,p\right) }\left( g\right) =\infty ,
\end{equation*}%
where $m,p,q$ are any three positive integers.
\end{corollary}

\begin{corollary}
\label{c8Q} Let $f\left( z\right) $, $g\left( z\right) $ and $h\left(
z\right) $ be any three entire functions of $n$- complex variables and $D$
be a bounded complete $n$-circular domain with center at origin in $%
\mathbb{C}
^{n}.$ Also let $g\left( z\right) $ be an entire function with relative
index-pair $\left( m,p\right) $, Then%
\begin{eqnarray*}
\left( i\right) ~\rho _{g}^{\left( p,q\right) }\left( f\right) &=&0~\text{%
when }\rho _{h}^{\left( m,q\right) }\left( f\right) =0~, \\
\left( ii\right) ~\lambda _{g}^{\left( p,q\right) }\left( f\right) &=&0~%
\text{when }\lambda _{h}^{\left( m,q\right) }\left( f\right) =0~, \\
\left( iii\right) ~\rho _{g}^{\left( p,q\right) }\left( f\right) &=&\infty ~%
\text{when }\rho _{h}^{\left( m,q\right) }\left( f\right) =\infty
\end{eqnarray*}%
and%
\begin{equation*}
\left( iv\right) ~\lambda _{g}^{\left( p,q\right) }\left( f\right) =\infty ~%
\text{when }\lambda _{h}^{\left( m,q\right) }\left( f\right) =\infty ,
\end{equation*}%
where $m,p,q$ are any three positive integers.
\end{corollary}

\begin{remark}
\label{r2} Under the same conditions of Theorem \ref{l2}, one may write $%
\rho _{g}^{\left( p,q\right) }\left( f\right) =\frac{\rho _{h}^{\left(
m,q\right) }\left( f\right) }{\rho _{h}^{\left( m,p\right) }\left( g\right) }
$\ and $\lambda _{g}^{\left( p,q\right) }\left( f\right) =\frac{\lambda
_{h}^{\left( m,q\right) }\left( f\right) }{\lambda _{h}^{\left( m,p\right)
}\left( g\right) }$ when $\lambda _{h}^{\left( m,p\right) }\left( g\right)
=\rho _{h}^{\left( m,p\right) }\left( g\right) .$ Similarly $\rho
_{g}^{\left( p,q\right) }\left( f\right) =\frac{\lambda _{h}^{\left(
m,q\right) }\left( f\right) }{\lambda _{h}^{\left( m,p\right) }\left(
g\right) }$ and\ $\lambda _{g}^{\left( p,q\right) }\left( f\right) =\frac{%
\rho _{h}^{\left( m,q\right) }\left( f\right) }{\rho _{h}^{\left( m,p\right)
}\left( g\right) }$ when $\lambda _{h}^{\left( m,q\right) }\left( f\right)
=\rho _{h}^{\left( m,q\right) }\left( f\right) .$
\end{remark}

\qquad Next we prove our theorem based on $(p,q)${\small -th relative
Gol'dberg type} and $(p,q)${\small -th relative Gol'dberg }weak type of
entire functions of $n$-complex variables

\begin{theorem}
{\normalsize \label{t1}} Let $f\left( z\right) $ and $g\left( z\right) $ be
any two entire functions of $n$- complex variables with index-pair $\left(
m,q\right) $ and $\left( m,p\right) ,$ respectively, where $m\geq q\geq 1$
and $m\geq p\geq 1$ and $D$ be a bounded complete $n$-circular domain with
center at origin in $%
\mathbb{C}
^{n}.$ Then{\normalsize 
\begin{equation*}
\max \left\{ \left[ \frac{\overline{\Delta }_{f,D}\left( m,q\right) }{\tau
_{g,D}\left( m,p\right) }\right] ^{\frac{1}{\lambda _{g}\left( m,p\right) }},%
\left[ \frac{\Delta _{f,D}\left( m,q\right) }{\overline{\tau }_{g,D}\left(
m,p\right) }\right] ^{\frac{1}{\lambda _{g}\left( m,p\right) }}\right\} \leq
\Delta _{g,D}^{\left( p,q\right) }\left( f\right) \leq \left[ \frac{\Delta
_{f,D}\left( m,q\right) }{\overline{\Delta }_{g,D}\left( m,p\right) }\right]
^{\frac{1}{\rho _{g}\left( m,p\right) }}~.
\end{equation*}%
}
\end{theorem}

\begin{proof}
{\normalsize From the definitions of }$\Delta _{f,D}\left( m,q\right) $%
{\normalsize \ and }$\overline{\Delta }_{f,D}\left( m,q\right) ${\normalsize 
$,$ we have for all sufficiently large values of }${\normalsize R}$%
{\normalsize \ that%
\begin{eqnarray}
M_{f,D}\left( R\right) &\leq &\exp ^{\left[ m-1\right] }\left[ \left( \Delta
_{f,D}\left( m,q\right) +\varepsilon \right) \left[ \log ^{\left[ q-1\right]
}R\right] ^{\rho _{f}\left( m,q\right) }\right] ,  \label{1} \\
M_{f,D}\left( R\right) &\geq &\exp ^{\left[ m-1\right] }\left[ \left( 
\overline{\Delta }_{f,D}\left( m,q\right) -\varepsilon \right) \left[ \log ^{%
\left[ q-1\right] }R\right] ^{\rho _{f}\left( m,q\right) }\right]  \label{2}
\end{eqnarray}%
and also for a sequence of values of }$R${\normalsize \ tending to infinity,
we get that%
\begin{eqnarray}
M_{f,D}\left( R\right) &\geq &\exp ^{\left[ m-1\right] }\left[ \left( \Delta
_{f,D}\left( m,q\right) -\varepsilon \right) \left[ \log ^{\left[ q-1\right]
}R\right] ^{\rho _{f}\left( m,q\right) }\right] ,  \label{3} \\
M_{f,D}\left( R\right) &\leq &\exp ^{\left[ m-1\right] }\left[ \left( 
\overline{\Delta }_{f,D}\left( m,q\right) +\varepsilon \right) \left[ \log ^{%
\left[ q-1\right] }R\right] ^{\rho _{f}\left( m,q\right) }\right] ~.
\label{4}
\end{eqnarray}%
}

\qquad {\normalsize Similarly from the definitions of }${\normalsize \Delta }%
_{g,D}\left( m,p\right) ${\normalsize \ and $\overline{\Delta }_{g,D}\left(
m,p\right) ,$ it follows for all sufficiently large values of }$R$%
{\normalsize \ that%
\begin{eqnarray}
M_{g,D}\left( R\right) &\leq &\exp ^{\left[ m-1\right] }\left[ \left( \Delta
_{g,D}\left( m,p\right) +\varepsilon \right) \left[ \log ^{\left[ p-1\right]
}R\right] ^{\rho _{g}\left( m,p\right) }\right]  \notag \\
i.e.,~R &\leq &M_{g,D}^{-1}\left[ \exp ^{\left[ m-1\right] }\left[ \left(
\Delta _{g,D}\left( m,p\right) +\varepsilon \right) \left[ \log ^{\left[ p-1%
\right] }R\right] ^{\rho _{g}\left( m,p\right) }\right] \right]  \notag \\
i.e.,~M_{g,D}^{-1}\left( R\right) &\geq &\left[ \exp ^{\left[ p-1\right]
}\left( \frac{\log ^{\left[ m-1\right] }R}{\left( \Delta _{g,D}\left(
m,p\right) +\varepsilon \right) }\right) ^{\frac{1}{\rho _{g}\left(
m,p\right) }}\right] ~\text{and}  \label{5}
\end{eqnarray}%
\begin{equation}
M_{g,D}^{-1}\left( R\right) \leq \left[ \exp ^{\left[ p-1\right] }\left( 
\frac{\log ^{\left[ m-1\right] }R}{\left( \overline{\Delta }_{g,D}\left(
m,p\right) -\varepsilon \right) }\right) ^{\frac{1}{\rho _{g}\left(
m,p\right) }}\right] ~.  \label{6}
\end{equation}%
}

\qquad {\normalsize Also for a sequence of values of }${\normalsize R}$%
{\normalsize \ tending to infinity, we obtain that%
\begin{equation}
M_{g,D}^{-1}\left( R\right) \leq \left[ \exp ^{\left[ p-1\right] }\left( 
\frac{\log ^{\left[ m-1\right] }R}{\left( \Delta _{g,D}\left( m,p\right)
-\varepsilon \right) }\right) ^{\frac{1}{\rho _{g}\left( m,p\right) }}\right]
~\text{\ and}  \label{7}
\end{equation}%
\begin{equation}
M_{g,D}^{-1}\left( R\right) \geq \left[ \exp ^{\left[ p-1\right] }\left( 
\frac{\log ^{\left[ m-1\right] }R}{\left( \overline{\Delta }_{g,D}\left(
m,p\right) +\varepsilon \right) }\right) ^{\frac{1}{\rho _{g}\left(
m,p\right) }}\right] ~.~\ \ \ \ \   \label{8}
\end{equation}%
}

\qquad {\normalsize From the definitions of $\overline{\tau }$}$_{f,D}\left(
m,q\right) ${\normalsize \ and $\tau $}$_{f,D}\left( m,q\right) $%
{\normalsize , we have for all sufficiently large values of }$R${\normalsize %
\ that%
\begin{eqnarray}
M_{f.D}\left( R\right) &\leq &\exp ^{\left[ m-1\right] }\left[ \left( 
\overline{\tau }_{f,D}\left( m,q\right) +\varepsilon \right) \left[ \log ^{%
\left[ q-1\right] }R\right] ^{\lambda _{f}\left( m,q\right) }\right] ,
\label{9} \\
M_{f,D}\left( R\right) &\geq &\exp ^{\left[ m-1\right] }\left[ \left( \tau
_{f,D}\left( m,q\right) -\varepsilon \right) \left[ \log ^{\left[ q-1\right]
}R\right] ^{\lambda _{f}\left( m,q\right) }\right]  \label{10}
\end{eqnarray}%
and also for a sequence of values of }${\normalsize R}${\normalsize \
tending to infinity, we get that%
\begin{eqnarray}
M_{f,D}\left( R\right) &\geq &\exp ^{\left[ m-1\right] }\left[ \left( 
\overline{\tau }_{f.D}\left( m,q\right) -\varepsilon \right) \left[ \log ^{%
\left[ q-1\right] }R\right] ^{\lambda _{f}\left( m,q\right) }\right] ,
\label{11} \\
M_{f,D}\left( R\right) &\leq &\exp ^{\left[ m-1\right] }\left[ \left( \tau
_{f,D}\left( m,q\right) +\varepsilon \right) \left[ \log ^{\left[ q-1\right]
}R\right] ^{\lambda _{f}\left( m,q\right) }\right] ~.  \label{12}
\end{eqnarray}%
}

\qquad {\normalsize Similarly from the definitions of $\overline{\tau }$}$%
_{g,D}\left( m,p\right) ${\normalsize \ and $\tau _{g,D}\left( m,p\right) ,$
it follows for all sufficiently large values of }${\normalsize R}$%
{\normalsize \ that%
\begin{eqnarray}
M_{g,D}\left( R\right) &\leq &\exp ^{\left[ m-1\right] }\left[ \left( 
\overline{\tau }_{g,D}\left( m,p\right) +\varepsilon \right) \left[ \log ^{%
\left[ p-1\right] }R\right] ^{\lambda _{g}\left( m,p\right) }\right]  \notag
\\
i.e.,~R &\leq &M_{g,D}^{-1}\left[ \exp ^{\left[ m-1\right] }\left[ \left( 
\overline{\tau }_{g,D}\left( m,p\right) +\varepsilon \right) \left[ \log ^{%
\left[ p-1\right] }R\right] ^{\lambda _{g}\left( m,p\right) }\right] \right]
\notag \\
i.e.,~M_{g,D}^{-1}\left( R\right) &\geq &\left[ \exp ^{\left[ p-1\right]
}\left( \frac{\log ^{\left[ m-1\right] }R}{\left( \overline{\tau }%
_{g,D}\left( m,p\right) +\varepsilon \right) }\right) ^{\frac{1}{\lambda
_{g}\left( m,p\right) }}\right] ~\text{and}  \label{13}
\end{eqnarray}%
\begin{equation}
M_{g,D}^{-1}\left( R\right) \leq \left[ \exp ^{\left[ p-1\right] }\left( 
\frac{\log ^{\left[ m-1\right] }R}{\left( \tau _{g,D}\left( m,p\right)
-\varepsilon \right) }\right) ^{\frac{1}{\lambda _{g}\left( m,p\right) }}%
\right] ~.~\ \ \ \ \ \ \ \ \ \ \ \   \label{14}
\end{equation}%
}

\qquad {\normalsize Also for a sequence of values of }${\normalsize R}$%
{\normalsize \ tending to infinity, we obtain that%
\begin{equation}
M_{g,D}^{-1}\left( R\right) \leq \left[ \exp ^{\left[ p-1\right] }\left( 
\frac{\log ^{\left[ m-1\right] }R}{\left( \overline{\tau }_{g,D}\left(
m,p\right) -\varepsilon \right) }\right) ^{\frac{1}{\lambda _{g}\left(
m,p\right) }}\right] ~\text{\ and}  \label{15}
\end{equation}%
\begin{equation}
M_{g,D}^{-1}\left( R\right) \geq \left[ \exp ^{\left[ p-1\right] }\left( 
\frac{\log ^{\left[ m-1\right] }R}{\left( \tau _{g,D}\left( m,p\right)
+\varepsilon \right) }\right) ^{\frac{1}{\lambda _{g}\left( m,p\right) }}%
\right] ~.~\ \ \ \ \   \label{16}
\end{equation}%
}

\qquad {\normalsize Now from $\left( \ref{3}\right) $ and in view of $\left( %
\ref{13}\right) $, we get for a sequence of values of }${\normalsize R}$%
{\normalsize \ tending to infinity that%
\begin{equation*}
M_{g,D}^{-1}M_{f,D}\left( R\right) \geq M_{g,D}^{-1}\left[ \exp ^{\left[ m-1%
\right] }\left[ \left( \Delta _{f,D}\left( m,q\right) -\varepsilon \right) %
\left[ \log ^{\left[ q-1\right] }R\right] ^{\rho _{f}\left( m,q\right) }%
\right] \right]
\end{equation*}%
\begin{equation*}
i.e.,~M_{g,D}^{-1}M_{f,D}\left( R\right) \geq ~\ \ \ \ \ \ \ \ \ \ \ \ \ \ \
\ \ \ \ \ \ \ \ \ \ \ \ \ \ \ \ \ \ \ \ \ \ \ \ \ \ \ \ \ \ \ \ \ \ \ \ \ \
\ \ \ \ \ \ \ \ \ \ \ \ \ \ \ \ \ \ \ \ \ \ \ \ \ \ \ \ \ \ \ \ \ \ \ \ \ \
\ \ \ \ \ \ \ \ \ \ \ \ \ 
\end{equation*}%
}%
\begin{equation*}
~\ \ \ \ \ \ \ \ \ \ \ \ \ \ \ \ \ \ \left[ \exp ^{\left[ p-1\right] }\left( 
\frac{\log ^{\left[ m-1\right] }\exp ^{\left[ m-1\right] }\left[ \left(
\Delta _{f,D}\left( m,q\right) -\varepsilon \right) \left[ \log ^{\left[ q-1%
\right] }R\right] ^{\rho _{f}\left( m,q\right) }\right] }{\left( 
{\normalsize \overline{\tau }}_{g,D}\left( m,p\right) +\varepsilon \right) }%
\right) ^{\frac{1}{\lambda _{g}\left( m,p\right) }}\right]
\end{equation*}%
{\normalsize 
\begin{equation}
i.e.,~\log ^{\left[ p-1\right] }M_{g,D}^{-1}M_{f,D}\left( R\right) \geq %
\left[ \frac{\left( \Delta _{f,D}\left( m,q\right) -\varepsilon \right) }{%
\left( \overline{\tau }_{g,D}\left( m,p\right) +\varepsilon \right) }\right]
^{\frac{1}{\lambda _{g}\left( m,p\right) }}\cdot \left[ \log ^{\left[ q-1%
\right] }R\right] ^{\frac{\rho _{f}\left( m,q\right) }{\lambda _{g}\left(
m,p\right) }}~.  \notag
\end{equation}%
}

\qquad {\normalsize Since in view of Theorem \ref{l1}, }$\frac{\rho
_{f}\left( m,q\right) }{\lambda _{g}\left( m,p\right) }${\normalsize $\geq $}%
$\rho _{g}^{\left( p,q\right) }\left( f\right) ${\normalsize \ and as $%
\varepsilon \left( >0\right) $ is arbitrary, therefore it follows from above
that%
\begin{align}
\overline{\underset{R\rightarrow +\infty }{\lim }}\frac{\log ^{\left[ p-1%
\right] }M_{g,D}^{-1}M_{f,D}\left( R\right) }{\left[ \log ^{\left[ q-1\right]
}R\right] ^{^{\rho _{g}^{\left( p,q\right) }\left( f\right) }}}& \geq \left[ 
\frac{\Delta _{f,D}\left( m,q\right) }{\overline{\tau }_{g,D}\left(
m,p\right) }\right] ^{\frac{1}{\lambda _{g}\left( m,p\right) }}  \notag \\
i.e.,~\Delta _{g,D}^{\left( p,q\right) }\left( f\right) & \geq \left[ \frac{%
\Delta _{f,D}\left( m,q\right) }{\overline{\tau }_{g,D}\left( m,p\right) }%
\right] ^{\frac{1}{\lambda _{g}\left( m,p\right) }}~.  \label{17}
\end{align}%
}

\qquad {\normalsize Similarly from $\left( \ref{2}\right) $ and in view of $%
\left( \ref{16}\right) $, it follows for a sequence of values of }$%
{\normalsize R}${\normalsize \ tending to infinity that%
\begin{equation*}
M_{g,D}^{-1}M_{f,D}\left( R\right) \geq M_{g,D}^{-1}\left[ \exp ^{\left[ m-1%
\right] }\left[ \left( \overline{\Delta }_{f,D}\left( m,q\right)
-\varepsilon \right) \left[ \log ^{\left[ q-1\right] }R\right] ^{\rho
_{f}\left( m,q\right) }\right] \right]
\end{equation*}%
}%
\begin{equation*}
i.e.,~M_{g,D}^{-1}M_{f,D}\left( R\right) \geq ~\ \ \ \ \ \ \ \ \ \ \ \ \ \ \
\ \ \ \ \ \ \ \ \ \ \ \ \ \ \ \ \ \ \ \ \ \ \ \ \ \ \ \ \ \ \ \ \ \ \ \ \ \
\ \ \ \ \ \ \ \ \ \ \ \ \ \ \ \ \ \ \ \ \ \ \ \ \ \ \ \ \ \ \ \ \ \ \ \ \ \
\ \ \ \ \ \ \ \ \ \ \ \ \ 
\end{equation*}%
\begin{equation*}
~\ \ \ \ \ \ \ \ \ \ \ \ \ \ \ \ \ \ \left[ \exp ^{\left[ p-1\right] }\left( 
\frac{\log ^{\left[ m-1\right] }\exp ^{\left[ m-1\right] }\left[ \left( 
\overline{\Delta }_{f,D}\left( m,q\right) -\varepsilon \right) \left[ \log ^{%
\left[ q-1\right] }R\right] ^{\rho _{f}\left( m,q\right) }\right] }{\left( 
{\normalsize \tau _{g,D}\left( m,p\right) }+\varepsilon \right) }\right) ^{%
\frac{1}{\lambda _{g}\left( m,p\right) }}\right]
\end{equation*}%
{\normalsize 
\begin{equation}
i.e.,~\log ^{\left[ p-1\right] }M_{g,D}^{-1}M_{f,D}\left( R\right) \geq %
\left[ \frac{\left( \overline{\Delta }_{f,D}\left( m,q\right) -\varepsilon
\right) }{\left( \tau _{g,D}\left( m,p\right) +\varepsilon \right) }\right]
^{\frac{1}{\lambda _{g}\left( m,p\right) }}\cdot \left[ \log ^{\left[ q-1%
\right] }R\right] ^{\frac{\rho _{f}\left( m,q\right) }{\lambda _{g}\left(
m,p\right) }}~.  \notag
\end{equation}%
}

\qquad {\normalsize Since in view of Theorem \ref{l1}, it follows that }$%
\frac{\rho _{f}\left( m,q\right) }{\lambda _{g}\left( m,p\right) }$%
{\normalsize $\geq \rho _{g}^{\left( p,q\right) }\left( f\right) .$ Also $%
\varepsilon \left( >0\right) $ is arbitrary, so we get from above that%
\begin{align}
\overline{\underset{R\rightarrow +\infty }{\lim }}\frac{\log ^{\left[ p-1%
\right] }M_{g,D}^{-1}M_{f}\left( R\right) }{\left[ \log ^{\left[ q-1\right]
}R\right] ^{^{\rho _{g}^{\left( p,q\right) }\left( f\right) }}}& \geq \left[ 
\frac{\overline{\Delta }_{f,D}\left( m,q\right) }{\tau _{g,D}\left(
m,p\right) }\right] ^{\frac{1}{\lambda _{g}\left( m,p\right) }}  \notag \\
i.e.,~\Delta _{g,D}^{\left( p,q\right) }\left( f\right) & \geq \left[ \frac{%
\overline{\Delta }_{f,D}\left( m,q\right) }{\tau _{g,D}\left( m,p\right) }%
\right] ^{\frac{1}{\lambda _{g}\left( m,p\right) }}~.  \label{18}
\end{align}%
}

\qquad {\normalsize Again in view of $\left( \ref{6}\right) $, we have from $%
\left( \ref{1}\right) $ for all sufficiently large values of }${\normalsize R%
}${\normalsize \ that%
\begin{equation*}
M_{g,D}^{-1}M_{f,D}\left( R\right) \leq M_{g,D}^{-1}\left[ \exp ^{\left[ m-1%
\right] }\left[ \left( \Delta _{f,D}\left( m,q\right) +\varepsilon \right) %
\left[ \log ^{\left[ q-1\right] }R\right] ^{\rho _{f}\left( m,q\right) }%
\right] \right]
\end{equation*}%
}%
\begin{equation*}
i.e.,~M_{g,D}^{-1}M_{f,D}\left( R\right) \leq ~\ \ \ \ \ \ \ \ \ \ \ \ \ \ \
\ \ \ \ \ \ \ \ \ \ \ \ \ \ \ \ \ \ \ \ \ \ \ \ \ \ \ \ \ \ \ \ \ \ \ \ \ \
\ \ \ \ \ \ \ \ \ \ \ \ \ \ \ \ \ \ \ \ \ \ \ \ \ \ \ \ \ \ \ \ \ \ \ \ \ \
\ \ \ \ \ \ \ \ \ \ \ \ \ 
\end{equation*}%
\begin{equation*}
~\ \ \ \ \ \ \ \ \ \ \ \ \ \ \ \ \ \ \left[ \exp ^{\left[ p-1\right] }\left( 
\frac{\log ^{\left[ m-1\right] }\exp ^{\left[ m-1\right] }\left[ \left(
\Delta _{f,D}\left( m,q\right) +\varepsilon \right) \left[ \log ^{\left[ q-1%
\right] }R\right] ^{\rho _{f}\left( m,q\right) }\right] }{\left( \overline{%
{\normalsize \Delta }}_{g,D}\left( m,p\right) -\varepsilon \right) }\right)
^{\frac{1}{\rho _{g}\left( m,p\right) }}\right]
\end{equation*}%
{\normalsize 
\begin{equation}
i.e.,~\log ^{\left[ p-1\right] }M_{g,D}^{-1}M_{f,D}\left( R\right) \leq %
\left[ \frac{\left( \Delta _{f,D}\left( m,q\right) +\varepsilon \right) }{%
\left( \overline{\Delta }_{g,D}\left( m,p\right) -\varepsilon \right) }%
\right] ^{\frac{1}{\rho _{g}\left( m,p\right) }}\cdot \left[ \log ^{\left[
q-1\right] }R\right] ^{\frac{\rho _{f}\left( m,q\right) }{\rho _{g}\left(
m,p\right) }}~.  \label{20}
\end{equation}%
}

\qquad {\normalsize As in view of Theorem \ref{l1}, it follows that }$\frac{%
\rho _{f}\left( m,q\right) }{\rho _{g}\left( m,p\right) }${\normalsize $\leq 
$}$\rho _{g}^{\left( p,q\right) }\left( f\right) ${\normalsize \ Since $%
\varepsilon \left( >0\right) $ is arbitrary, we get from $\left( \ref{20}%
\right) $ that%
\begin{align}
\overline{\underset{R\rightarrow +\infty }{\lim }}\frac{\log ^{\left[ p-1%
\right] }M_{g,D}^{-1}M_{f,D}\left( R\right) }{\left[ \log ^{\left[ q-1\right]
}R\right] ^{^{\rho _{g}^{\left( p,q\right) }\left( f\right) }}}& \leq \left[ 
\frac{\Delta _{f,D}\left( m,q\right) }{\overline{\Delta }_{g,D}\left(
m,p\right) }\right] ^{\frac{1}{\rho _{g}\left( m,p\right) }}  \notag \\
i.e.,~\Delta _{g,D}^{\left( p,q\right) }\left( f\right) & \leq \left[ \frac{%
\Delta _{f,D}\left( m,q\right) }{\overline{\Delta }_{g,D}\left( m,p\right) }%
\right] ^{\frac{1}{\rho _{g}\left( m,p\right) }}~.  \label{21}
\end{align}%
}

\qquad {\normalsize Thus the theorem follows from $\left( \ref{17}\right) $, 
$\left( \ref{18}\right) $ and $\left( \ref{21}\right) $. }
\end{proof}

{\normalsize \qquad The conclusion of the following corollary can be carried
out from $\left( \ref{6}\right) $ and $\left( \ref{9}\right) $; $\left( \ref%
{9}\right) $ and $\left( \ref{14}\right) $ respectively after applying the
same technique of Theorem \ref{t1} and with the help of Theorem \ref{l1}.
Therefore its proof is omitted. }

\begin{corollary}
{\normalsize \label{c1} }Let $f\left( z\right) $ and $g\left( z\right) $ be
any two entire functions of $n$- complex variables with index-pair $\left(
m,q\right) $ and $\left( m,p\right) ,$ respectively, where $m\geq q\geq 1$
and $m\geq p\geq 1$ and $D$ be a bounded complete $n$-circular domain with
center at origin in $%
\mathbb{C}
^{n}.$ Then{\normalsize 
\begin{equation*}
\Delta _{g,D}^{\left( p,q\right) }\left( f\right) \leq \min \left\{ \left[ 
\frac{\overline{\tau }_{f,D}\left( m,q\right) }{\tau _{g,D}\left( m,p\right) 
}\right] ^{\frac{1}{\lambda _{g}\left( m,p\right) }},\left[ \frac{\overline{%
\tau }_{f,D}\left( m,q\right) }{\overline{\Delta }_{g,D}\left( m,p\right) }%
\right] ^{\frac{1}{\rho _{g}\left( m,p\right) }}\right\} ~.
\end{equation*}%
}
\end{corollary}

{\normalsize \qquad Similarly in the line of Theorem \ref{t1} and with the
help of Theorem \ref{l1}, one may easily carried out the following theorem
from pairwise inequalities numbers $\left( \ref{10}\right) $ and $\left( \ref%
{13}\right) ;$ $\left( \ref{7}\right) $ and $\left( \ref{9}\right) $; $%
\left( \ref{6}\right) $ and $\left( \ref{12}\right) $ respectively and
therefore its proofs is omitted: }

\begin{theorem}
{\normalsize \label{t4} }Let $f\left( z\right) $ and $g\left( z\right) $ be
any two entire functions of $n$- complex variables with index-pair $\left(
m,q\right) $ and $\left( m,p\right) ,$ respectively, where $m\geq q\geq 1$
and $m\geq p\geq 1$ and $D$ be a bounded complete $n$-circular domain with
center at origin in $%
\mathbb{C}
^{n}.$ Then{\normalsize 
\begin{equation*}
\left[ \frac{\tau _{f,D}\left( m,q\right) }{\overline{\tau }_{g,D}\left(
m,p\right) }\right] ^{\frac{1}{\lambda _{g}\left( m,p\right) }}\leq \tau
_{g,D}^{\left( p,q\right) }\left( f\right) \leq \min \left\{ \left[ \frac{%
\tau _{f,D}\left( m,q\right) }{\overline{\Delta }_{g,D}\left( m,p\right) }%
\right] ^{\frac{1}{\rho _{g}\left( m,p\right) }},\left[ \frac{\overline{\tau 
}_{f,D}\left( m,q\right) }{\Delta _{g,D}\left( m,p\right) }\right] ^{\frac{1%
}{\rho _{g}\left( m,p\right) }}\right\} ~.
\end{equation*}%
}
\end{theorem}

\begin{corollary}
{\normalsize \label{c4} }Let $f\left( z\right) $ and $g\left( z\right) $ be
any two entire functions of $n$- complex variables with index-pair $\left(
m,q\right) $ and $\left( m,p\right) ,$ respectively, where $m\geq q\geq 1$
and $m\geq p\geq 1$ and $D$ be a bounded complete $n$-circular domain with
center at origin in $%
\mathbb{C}
^{n}.$ Then{\normalsize 
\begin{equation*}
\tau _{g,D}^{\left( p,q\right) }\left( f\right) \geq \max \left\{ \left[ 
\frac{\overline{\Delta }_{f,D}\left( m,q\right) }{\Delta _{g,D}\left(
m,p\right) }\right] ^{\frac{1}{\rho _{g}\left( m,p\right) }},\left[ \frac{%
\overline{\Delta }_{f,D}\left( m,q\right) }{\overline{\tau }_{g,D}\left(
m,p\right) }\right] ^{\frac{1}{\lambda _{g}\left( m,p\right) }}\right\} ~.
\end{equation*}%
}
\end{corollary}

{\normalsize \qquad With the help of Theorem \ref{l1}, the conclusion of the
above corollary can be carry out from $\left( \ref{2}\right) ,$ $\left( \ref%
{5}\right) $ and $\left( \ref{2}\right) ,\left( \ref{13}\right) $
respectively after applying the same technique of Theorem \ref{t1} and
therefore its proof is omitted. }

\begin{theorem}
{\normalsize \label{t2} }Let $f\left( z\right) $ and $g\left( z\right) $ be
any two entire functions of $n$- complex variables with index-pair $\left(
m,q\right) $ and $\left( m,p\right) ,$ respectively, where $m\geq q\geq 1$
and $m\geq p\geq 1$ and $D$ be a bounded complete $n$-circular domain with
center at origin in $%
\mathbb{C}
^{n}.$ Then{\normalsize 
\begin{equation*}
\left[ \frac{\overline{\Delta }_{f,D}\left( m,q\right) }{\overline{\tau }%
_{g,D}\left( m,p\right) }\right] ^{\frac{1}{\lambda _{g}\left( m,p\right) }%
}\leq \overline{\Delta }_{g,D}^{\left( p,q\right) }\left( f\right) \leq \min
\left\{ \left[ \frac{\overline{\Delta }_{f,D}\left( m,q\right) }{\overline{%
\Delta }_{g,D}\left( m,p\right) }\right] ^{\frac{1}{\rho _{g}\left(
m,p\right) }},\left[ \frac{\Delta _{f,D}\left( m,q\right) }{\Delta
_{g,D}\left( m,p\right) }\right] ^{\frac{1}{\rho _{g}\left( m,p\right) }%
}\right\} ~.
\end{equation*}%
}
\end{theorem}

\begin{proof}
{\normalsize From $\left( \ref{2}\right) $ and in view of $\left( \ref{13}%
\right) $, we get for all sufficiently large values of }${\normalsize R}$%
{\normalsize \ that%
\begin{equation*}
M_{g,D}^{-1}M_{f,D}\left( R\right) \geq M_{g}^{-1}\left[ \exp ^{\left[ m-1%
\right] }\left[ \left( \overline{\Delta }_{f,D}\left( m,q\right)
-\varepsilon \right) \left[ \log ^{\left[ q-1\right] }R\right] ^{\rho
_{f}\left( m,q\right) }\right] \right]
\end{equation*}%
}%
\begin{equation*}
i.e.,~M_{g}^{-1}M_{f}\left( R\right) \geq ~\ \ \ \ \ \ \ \ \ \ \ \ \ \ \ \ \
\ \ \ \ \ \ \ \ \ \ \ \ \ \ \ \ \ \ \ \ \ \ \ \ \ \ \ \ \ \ \ \ \ \ \ \ \ \
\ \ \ \ \ \ \ \ \ \ \ \ \ \ \ \ \ \ \ \ \ \ \ \ \ \ \ \ \ \ \ \ \ \ \ \ \ \
\ \ \ \ \ \ \ \ \ \ \ 
\end{equation*}%
\begin{equation*}
~\ \ \ \ \ \ \ \ \ \ \ \ \ \ \ \ \ \ \left[ \exp ^{\left[ p-1\right] }\left( 
\frac{\log ^{\left[ m-1\right] }\exp ^{\left[ m-1\right] }\left[ \left( 
\overline{\Delta }_{f,D}\left( m,q\right) -\varepsilon \right) \left[ \log ^{%
\left[ q-1\right] }R\right] ^{\rho _{f}\left( m,q\right) }\right] }{\left( 
{\normalsize \overline{\tau }}_{g,D}\left( m,p\right) +\varepsilon \right) }%
\right) ^{\frac{1}{\lambda _{g}\left( m,p\right) }}\right]
\end{equation*}%
{\normalsize 
\begin{equation}
i.e.,~\log ^{\left[ p-1\right] }M_{g,D}^{-1}M_{f,D}\left( R\right) \geq %
\left[ \frac{\left( \overline{\Delta }_{f,D}\left( m,q\right) -\varepsilon
\right) }{\left( \overline{\tau }_{g,D}\left( m,p\right) +\varepsilon
\right) }\right] ^{\frac{1}{\lambda _{g}\left( m,p\right) }}\cdot \left[
\log ^{\left[ q-1\right] }R\right] ^{\frac{\rho _{f}\left( m,q\right) }{%
\lambda _{g}\left( m,p\right) }}~.  \notag
\end{equation}%
}

\qquad {\normalsize Now in view of Theorem \ref{l1}, it follows that }$\frac{%
\rho _{f}\left( m,q\right) }{\lambda _{g}\left( m,p\right) }${\normalsize $%
\geq \rho _{g}^{\left( p,q\right) }\left( f\right) .$ Since $\varepsilon
\left( >0\right) $ is arbitrary, we get from above that%
\begin{align}
\underset{R\rightarrow +\infty }{\underline{\lim }}\frac{\log ^{\left[ p-1%
\right] }M_{g,D}^{-1}M_{f,D}\left( R\right) }{\left[ \log ^{\left[ q-1\right]
}R\right] ^{^{\rho _{g}^{\left( p,q\right) }\left( f\right) }}}& \geq \left[ 
\frac{\overline{\Delta }_{f,D}\left( m,q\right) }{\overline{\tau }%
_{g,D}\left( m,p\right) }\right] ^{\frac{1}{\lambda _{g}\left( m,p\right) }}
\notag \\
i.e.,~\overline{\Delta }_{g,D}^{\left( p,q\right) }\left( f\right) & \geq %
\left[ \frac{\overline{\Delta }_{f,D}\left( m,q\right) }{\overline{\tau }%
_{g,D}\left( m,p\right) }\right] ^{\frac{1}{\lambda _{g}\left( m,p\right) }%
}~.  \label{23}
\end{align}%
}

\qquad {\normalsize Further in view of $\left( \ref{7}\right) ,$ we get from 
$\left( \ref{1}\right) $ for a sequence of values of }${\normalsize R}$%
{\normalsize \ tending to infinity that%
\begin{equation*}
M_{g,D}^{-1}M_{f,D}\left( R\right) \leq M_{g,D}^{-1}\left[ \exp ^{\left[ m-1%
\right] }\left[ \left( \Delta _{f,D}\left( m,q\right) +\varepsilon \right) %
\left[ \log ^{\left[ q-1\right] }R\right] ^{\rho _{f}\left( m,q\right) }%
\right] \right]
\end{equation*}%
}%
\begin{equation*}
i.e.,~M_{g,D}^{-1}M_{f,D}\left( R\right) \leq ~\ \ \ \ \ \ \ \ \ \ \ \ \ \ \
\ \ \ \ \ \ \ \ \ \ \ \ \ \ \ \ \ \ \ \ \ \ \ \ \ \ \ \ \ \ \ \ \ \ \ \ \ \
\ \ \ \ \ \ \ \ \ \ \ \ \ \ \ \ \ \ \ \ \ \ \ \ \ \ \ \ \ \ \ \ \ \ \ \ \ \
\ \ \ \ \ \ \ \ \ \ \ \ \ 
\end{equation*}%
\begin{equation*}
~\ \ \ \ \ \ \ \ \ \ \ \ \ \ \ \ \ \ \left[ \exp ^{\left[ p-1\right] }\left( 
\frac{\log ^{\left[ m-1\right] }\exp ^{\left[ m-1\right] }\left[ \left(
\Delta _{f,D}\left( m,q\right) +\varepsilon \right) \left[ \log ^{\left[ q-1%
\right] }R\right] ^{\rho _{f}\left( m,q\right) }\right] }{\left( 
{\normalsize \Delta }_{g,D}\left( m,p\right) -\varepsilon \right) }\right) ^{%
\frac{1}{\rho _{g}\left( m,p\right) }}\right]
\end{equation*}%
{\normalsize 
\begin{equation}
i.e.,~\log ^{\left[ p-1\right] }M_{g,D}^{-1}M_{f,D}\left( R\right) \leq %
\left[ \frac{\left( \Delta _{f,D}\left( m,q\right) +\varepsilon \right) }{%
\left( \Delta _{g,D}\left( m,p\right) -\varepsilon \right) }\right] ^{\frac{1%
}{\rho _{g}\left( m,p\right) }}\cdot \left[ \log ^{\left[ q-1\right] }R%
\right] ^{\frac{\rho _{f}\left( m,q\right) }{\rho _{g}\left( m,p\right) }}~.
\label{24}
\end{equation}%
}

\qquad {\normalsize Again as in view of Theorem \ref{l1}, }$\frac{\rho
_{f}\left( m,q\right) }{\rho _{g}\left( m,p\right) }${\normalsize $\leq $}$%
\rho _{g}^{\left( p,q\right) }\left( f\right) ${\normalsize \ and $%
\varepsilon \left( >0\right) $ is arbitrary, therefore we get from $\left( %
\ref{24}\right) $ that%
\begin{align}
\underset{R\rightarrow +\infty }{\underline{\lim }}\frac{\log ^{\left[ p-1%
\right] }M_{g,D}^{-1}M_{f,D}\left( R\right) }{\left[ \log ^{\left[ q-1\right]
}R\right] ^{^{\rho _{g}^{\left( p,q\right) }\left( f\right) }}}& \leq \left[ 
\frac{\Delta _{f,D}\left( m,q\right) }{\Delta _{g,D}\left( m,p\right) }%
\right] ^{\frac{1}{\rho _{g}\left( m,p\right) }}  \notag \\
i.e.,~\overline{\Delta }_{g,D}^{\left( p,q\right) }\left( f\right) & \leq %
\left[ \frac{\Delta _{f,D}\left( m,q\right) }{\Delta _{g,D}\left( m,p\right) 
}\right] ^{\frac{1}{\rho _{g}\left( m,p\right) }}~.  \label{25}
\end{align}%
}

\qquad {\normalsize Likewise from $\left( \ref{4}\right) $ and in view of $%
\left( \ref{6}\right) $, it follows for a sequence of values of }$R$%
{\normalsize \ tending to infinity that%
\begin{equation*}
M_{g,D}^{-1}M_{f,D}\left( R\right) \leq M_{g,D}^{-1}\left[ \exp ^{\left[ m-1%
\right] }\left[ \left( \overline{\Delta }_{f,D}\left( m,q\right)
+\varepsilon \right) \left[ \log ^{\left[ q-1\right] }R\right] ^{\rho
_{f}\left( m,q\right) }\right] \right]
\end{equation*}%
}%
\begin{equation*}
i.e.,~M_{g,D}^{-1}M_{f,D}\left( R\right) \leq ~\ \ \ \ \ \ \ \ \ \ \ \ \ \ \
\ \ \ \ \ \ \ \ \ \ \ \ \ \ \ \ \ \ \ \ \ \ \ \ \ \ \ \ \ \ \ \ \ \ \ \ \ \
\ \ \ \ \ \ \ \ \ \ \ \ \ \ \ \ \ \ \ \ \ \ \ \ \ \ \ \ \ \ \ \ \ \ \ \ \ \
\ \ \ \ \ \ \ \ \ \ \ \ \ 
\end{equation*}%
\begin{equation*}
~\ \ \ \ \ \ \ \ \ \ \ \ \ \ \ \ \ \ \left[ \exp ^{\left[ p-1\right] }\left( 
\frac{\log ^{\left[ m-1\right] }\exp ^{\left[ m-1\right] }\left[ \left( 
\overline{\Delta }_{f,D}\left( m,q\right) +\varepsilon \right) \left[ \log ^{%
\left[ q-1\right] }R\right] ^{\rho _{f}\left( m,q\right) }\right] }{\left( 
\overline{{\normalsize \Delta }}_{g,D}\left( m,p\right) -\varepsilon \right) 
}\right) ^{\frac{1}{\rho _{g}\left( m,p\right) }}\right]
\end{equation*}%
{\normalsize 
\begin{equation}
i.e.,~\log ^{\left[ p-1\right] }M_{g,D}^{-1}M_{f,D}\left( R\right) \leq %
\left[ \frac{\left( \overline{\Delta }_{f,D}\left( m,q\right) +\varepsilon
\right) }{\left( \overline{\Delta }_{g,D}\left( m,p\right) -\varepsilon
\right) }\right] ^{\frac{1}{\rho _{g}\left( m,p\right) }}\cdot \left[ \log ^{%
\left[ q-1\right] }R\right] ^{\frac{\rho _{f}\left( m,q\right) }{\rho
_{g}\left( m,p\right) }}~.  \label{26}
\end{equation}%
}

\qquad {\normalsize Analogously, we get from $\left( \ref{26}\right) $ that%
\begin{align}
\underset{r\rightarrow \infty }{\underline{\lim }}\frac{\log ^{\left[ p-1%
\right] }M_{g,D}^{-1}M_{f,D}\left( R\right) }{\left[ \log ^{\left[ q-1\right]
}R\right] ^{^{\rho _{g}^{\left( p,q\right) }\left( f\right) }}}& \leq \left[ 
\frac{\overline{\Delta }_{f,D}\left( m,q\right) }{\overline{\Delta }%
_{g,D}\left( m,p\right) }\right] ^{\frac{1}{\rho _{g}\left( m,p\right) }} 
\notag \\
i.e.,~~\overline{\Delta }_{g,D}^{\left( p,q\right) }\left( f\right) & \leq %
\left[ \frac{\overline{\Delta }_{f,D}\left( m,q\right) }{\overline{\Delta }%
_{g,D}\left( m,p\right) }\right] ^{\frac{1}{\rho _{g}\left( m,p\right) }},
\label{27}
\end{align}%
since in view of Theorem \ref{l1}, }$\frac{\rho _{f}\left( m,q\right) }{\rho
_{g}\left( m,p\right) }${\normalsize $\leq $}$\rho _{g}^{\left( p,q\right)
}\left( f\right) ${\normalsize \ and $\varepsilon \left( >0\right) $ is
arbitrary. }

\qquad {\normalsize Thus the theorem follows from $\left( \ref{23}\right) $, 
$\left( \ref{25}\right) $ and $\left( \ref{27}\right) $. }
\end{proof}

\begin{corollary}
{\normalsize \label{c2} }Let $f\left( z\right) $ and $g\left( z\right) $ be
any two entire functions of $n$- complex variables with index-pair $\left(
m,q\right) $ and $\left( m,p\right) ,$ respectively, where $m\geq q\geq 1$
and $m\geq p\geq 1$ and $D$ be a bounded complete $n$-circular domain with
center at origin in $%
\mathbb{C}
^{n}.$ Then{\normalsize 
\begin{equation*}
\overline{\Delta }_{g,D}^{\left( p,q\right) }\left( f\right) \leq ~\ \ \ \ \
\ \ \ \ \ \ \ \ \ \ \ \ \ \ \ \ \ \ \ \ \ \ \ \ \ \ \ \ \ \ \ \ \ \ \ \ \ \
\ \ \ \ \ \ \ \ \ \ \ \ \ \ \ \ \ \ \ \ \ \ \ \ \ \ \ \ \ \ \ \ \ \ \ \ \ \
\ \ \ \ \ \ \ \ \ \ \ \ \ \ \ \ \ \ \ \ \ \ \ \ \ \ \ \ \ \ \ \ \ \ \ \ \ \
\ \ 
\end{equation*}%
\begin{equation*}
\min \left\{ \left[ \frac{\tau _{f,D}\left( m,q\right) }{\tau _{g,D}\left(
m,p\right) }\right] ^{\frac{1}{\lambda _{g}\left( m,p\right) }},\left[ \frac{%
\overline{\tau }_{f,D}\left( m,q\right) }{\overline{\tau }_{g,D}\left(
m,p\right) }\right] ^{\frac{1}{\lambda _{g}\left( m,p\right) }},\left[ \frac{%
\overline{\tau }_{f,D}\left( m,q\right) }{\sigma _{g,D}\left( m,p\right) }%
\right] ^{\frac{1}{\rho _{g}\left( m,p\right) }},\left[ \frac{\tau
_{f,D}\left( m,q\right) }{\overline{\sigma }_{g,D}\left( m,p\right) }\right]
^{\frac{1}{\rho _{g}\left( m,p\right) }}\right\} ~.
\end{equation*}%
}
\end{corollary}

{\normalsize \qquad The conclusion of the above corollary can be carried out
from pairwise inequalities no $\left( \ref{6}\right) $ and $\left( \ref{12}%
\right) ;$ $\left( \ref{7}\right) $ and $\left( \ref{9}\right) ;$ $\left( %
\ref{12}\right) $ and $\left( \ref{14}\right) $; $\left( \ref{9}\right) $
and $\left( \ref{15}\right) $ respectively after applying the same technique
of Theorem \ref{t2} and with the help of Theorem \ref{l1}. Therefore its
proof is omitted. }

{\normalsize \qquad Similarly in the line of Theorem \ref{t1} and with the
help of Theorem \ref{l1}, one may easily carried out the following theorem
from pairwise inequalities no $\left( \ref{11}\right) $ and $\left( \ref{13}%
\right) ;$ $\left( \ref{10}\right) $ and $\left( \ref{16}\right) $; $\left( %
\ref{6}\right) $ and $\left( \ref{9}\right) $ respectively and therefore its
proofs is omitted: }

\begin{theorem}
{\normalsize \label{t3} }Let $f\left( z\right) $ and $g\left( z\right) $ be
any two entire functions of $n$- complex variables with index-pair $\left(
m,q\right) $ and $\left( m,p\right) ,$ respectively, where $m\geq q\geq 1$
and $m\geq p\geq 1$ and $D$ be a bounded complete $n$-circular domain with
center at origin in $%
\mathbb{C}
^{n}.$ Then{\normalsize 
\begin{equation*}
\max \left\{ \left[ \frac{\overline{\tau }_{f,D}\left( m,q\right) }{%
\overline{\tau }_{g,D}\left( m,p\right) }\right] ^{\frac{1}{\lambda
_{g}\left( m,p\right) }},\left[ \frac{\tau _{f,D}\left( m,q\right) }{\tau
_{g,D}\left( m,p\right) }\right] ^{\frac{1}{\lambda _{g}\left( m,p\right) }%
}\right\} \leq \overline{\tau }_{g,D}^{\left( p,q\right) }\left( f\right)
\leq \left[ \frac{\overline{\tau }_{f,D}\left( m,q\right) }{\overline{\Delta 
}_{g,D}\left( m,p\right) }\right] ^{\frac{1}{\rho _{g}\left( m,p\right) }}~.
\end{equation*}%
}
\end{theorem}

\begin{corollary}
{\normalsize \label{c3} }Let $f\left( z\right) $ and $g\left( z\right) $ be
any two entire functions of $n$- complex variables with index-pair $\left(
m,q\right) $ and $\left( m,p\right) ,$ respectively, where $m\geq q\geq 1$
and $m\geq p\geq 1$ and $D$ be a bounded complete $n$-circular domain with
center at origin in $%
\mathbb{C}
^{n}.$ Then{\normalsize 
\begin{equation*}
\overline{\tau }_{g,D}^{\left( p,q\right) }\left( f\right) \geq ~\ \ \ \ \ \
\ \ \ \ \ \ \ \ \ \ \ \ \ \ \ \ \ \ \ \ \ \ \ \ \ \ \ \ \ \ \ \ \ \ \ \ \ \
\ \ \ \ \ \ \ \ \ \ \ \ \ \ \ \ \ \ \ \ \ \ \ \ \ \ \ \ \ \ \ \ \ \ \ \ \ \
\ \ \ \ \ \ \ \ \ \ \ \ \ \ \ \ \ \ \ \ \ \ \ \ \ \ \ \ \ \ \ \ \ \ \ \ \ \
\ \ \ 
\end{equation*}%
\begin{equation*}
\max \left\{ \left[ \frac{\overline{\Delta }_{f,D}\left( m,q\right) }{%
\overline{\Delta }_{g,D}\left( m,p\right) }\right] ^{\frac{1}{\rho
_{g}\left( m,p\right) }},\left[ \frac{\Delta _{f,D}\left( m,q\right) }{%
\Delta _{g,D}\left( m,p\right) }\right] ^{\frac{1}{\rho _{g}\left(
m,p\right) }},\left[ \frac{\Delta _{f,D}\left( m,q\right) }{\overline{\tau }%
_{g,D}\left( m,p\right) }\right] ^{\frac{1}{\lambda _{g}\left( m,p\right) }},%
\left[ \frac{\overline{\Delta }_{f,D}\left( m,q\right) }{\tau _{g,D}\left(
m,p\right) }\right] ^{\frac{1}{\lambda _{g}\left( m,p\right) }}\right\} ~.
\end{equation*}%
}
\end{corollary}

{\normalsize \qquad The conclusion of the above corollary can be carried out
from pairwise inequalities no $\left( \ref{3}\right) $ and $\left( \ref{5}%
\right) ;$ $\left( \ref{2}\right) $ and $\left( \ref{8}\right) ;$ $\left( %
\ref{3}\right) $ and $\left( \ref{13}\right) $; $\left( \ref{2}\right) $ and 
$\left( \ref{16}\right) $ respectively after applying the same technique of
Theorem \ref{t2} and with the help of Theorem \ref{l1}. Therefore its proof
is omitted. }

\qquad Now we state the following theorems without their proofs as because
they can be derived easily using the same technique or with some easy
reasoning with the help of Remark \ref{r1} and therefore left to the readers.

\begin{theorem}
\label{t4.1} Let $f\left( z\right) $ and $g\left( z\right) $ be any two
entire functions of $n$- complex variables with index-pair $\left(
m,q\right) $ and $\left( m,p\right) ,$ respectively, where $m\geq q\geq 1$
and $m\geq p\geq 1$ and $D$ be a bounded complete $n$-circular domain with
center at origin in $%
\mathbb{C}
^{n}.$ Also let $g\left( z\right) $ is of regular $\left( m,p\right) $%
-Gol'dberg growth. Then%
\begin{multline*}
\left[ \frac{\overline{\Delta }_{f,D}\left( m,q\right) }{\Delta _{g,D}\left(
m,p\right) }\right] ^{\frac{1}{\rho _{g}\left( m,p\right) }}\leq \overline{%
\Delta }_{g,D}^{\left( p,q\right) }\left( f\right) \leq \min \left\{ \left[ 
\frac{\overline{\Delta }_{f,D}\left( m,q\right) }{\overline{\Delta }%
_{g,D}\left( m,p\right) }\right] ^{\frac{1}{\rho _{g}\left( m,p\right) }},%
\left[ \frac{\Delta _{f,D}\left( m,q\right) }{\Delta _{g,D}\left( m,p\right) 
}\right] ^{\frac{1}{\rho _{g}\left( m,p\right) }}\right\} \\
\leq \max \left\{ \left[ \frac{\overline{\Delta }_{f,D}\left( m,q\right) }{%
\overline{\Delta }_{g,D}\left( m,p\right) }\right] ^{\frac{1}{\rho
_{g}\left( m,p\right) }},\left[ \frac{\Delta _{f,D}\left( m,q\right) }{%
\Delta _{g,D}\left( m,p\right) }\right] ^{\frac{1}{\rho _{g}\left(
m,p\right) }}\right\} \leq \Delta _{g,D}^{\left( p,q\right) }\left( f\right)
\leq \left[ \frac{\Delta _{f,D}\left( m,q\right) }{\overline{\Delta }%
_{g,D}\left( m,p\right) }\right] ^{\frac{1}{\rho _{g}\left( m,p\right) }}
\end{multline*}%
and%
\begin{multline*}
\left[ \frac{\tau _{f,D}\left( m,q\right) }{\overline{\tau }_{g,D}\left(
m,p\right) }\right] ^{\frac{1}{\lambda _{g}\left( m,p\right) }}\leq \tau
_{g,D}^{\left( p,q\right) }\left( f\right) \leq \min \left\{ \left[ \frac{%
\tau _{f,D}\left( m,q\right) }{\tau _{g,D}\left( m,p\right) }\right] ^{\frac{%
1}{\lambda _{g}\left( m,p\right) }},\left[ \frac{\overline{\tau }%
_{f,D}\left( m,q\right) }{\overline{\tau }_{g,D}\left( m,p\right) }\right] ^{%
\frac{1}{\lambda _{g}\left( m,p\right) }}\right\} \\
\leq \max \left\{ \left[ \frac{\tau _{f,D}\left( m,q\right) }{\tau
_{g,D}\left( m,p\right) }\right] ^{\frac{1}{\lambda _{g}\left( m,p\right) }},%
\left[ \frac{\overline{\tau }_{f,D}\left( m,q\right) }{\overline{\tau }%
_{g,D}\left( m,p\right) }\right] ^{\frac{1}{\lambda _{g}\left( m,p\right) }%
}\right\} \leq \overline{\tau }_{g,D}^{\left( p,q\right) }\left( f\right)
\leq \left[ \frac{\overline{\tau }_{f}\left( m,q\right) }{\tau _{g}\left(
m,p\right) }\right] ^{\frac{1}{\lambda _{g}\left( m,p\right) }}~.
\end{multline*}
\end{theorem}

\begin{theorem}
\label{t4.2} Let $f\left( z\right) $ and $g\left( z\right) $ be any two
entire functions of $n$- complex variables with index-pair $\left(
m,q\right) $ and $\left( m,p\right) ,$ respectively, where $m\geq q\geq 1$
and $m\geq p\geq 1$ and $D$ be a bounded complete $n$-circular domain with
center at origin in $%
\mathbb{C}
^{n}.$ Also let $f\left( z\right) $ is of regular $\left( m,q\right) $%
-Gol'dberg growth. Then%
\begin{multline*}
\left[ \frac{\overline{\Delta }_{f,D}\left( m,q\right) }{\Delta _{g,D}\left(
m,p\right) }\right] ^{\frac{1}{\rho _{g}\left( m,p\right) }}\leq \tau
_{g,D}^{\left( p,q\right) }\left( f\right) \leq \min \left\{ \left[ \frac{%
\overline{\Delta }_{f,D}\left( m,q\right) }{\overline{\Delta }_{g,D}\left(
m,p\right) }\right] ^{\frac{1}{\rho _{g}\left( m,p\right) }},\left[ \frac{%
\Delta _{f,D}\left( m,q\right) }{\Delta _{g,D}\left( m,p\right) }\right] ^{%
\frac{1}{\rho _{g}\left( m,p\right) }}\right\} \\
\leq \max \left\{ \left[ \frac{\overline{\Delta }_{f,D}\left( m,q\right) }{%
\overline{\Delta }_{g,D}\left( m,p\right) }\right] ^{\frac{1}{\rho
_{g}\left( m,p\right) }},\left[ \frac{\Delta _{f,D}\left( m,q\right) }{%
\Delta _{g,D}\left( m,p\right) }\right] ^{\frac{1}{\rho _{g}\left(
m,p\right) }}\right\} \leq \overline{\tau }_{g,D}^{\left( p,q\right) }\left(
f\right) \leq \left[ \frac{\Delta _{f,D}\left( m,q\right) }{\overline{\Delta 
}_{g,D}\left( m,p\right) }\right] ^{\frac{1}{\rho _{g}\left( m,p\right) }}
\end{multline*}%
and 
\begin{multline*}
\left[ \frac{\tau _{f,D}\left( m,q\right) }{\overline{\tau }_{g,D}\left(
m,p\right) }\right] ^{\frac{1}{\lambda _{g}\left( m,p\right) }}\leq 
\overline{\Delta }_{g,D}^{\left( p,q\right) }\left( f\right) \leq \min
\left\{ \left[ \frac{\tau _{f,D}\left( m,q\right) }{\tau _{g,D}\left(
m,p\right) }\right] ^{\frac{1}{\lambda _{g}\left( m,p\right) }},\left[ \frac{%
\overline{\tau }_{f,D}\left( m,q\right) }{\overline{\tau }_{g,D}\left(
m,p\right) }\right] ^{\frac{1}{\lambda _{g}\left( m,p\right) }}\right\} \\
\leq \max \left\{ \left[ \frac{\tau _{f,D}\left( m,q\right) }{\tau
_{g,D}\left( m,p\right) }\right] ^{\frac{1}{\lambda _{g}\left( m,p\right) }},%
\left[ \frac{\overline{\tau }_{f,D}\left( m,q\right) }{\overline{\tau }%
_{g,D}\left( m,p\right) }\right] ^{\frac{1}{\lambda _{g}\left( m,p\right) }%
}\right\} \leq \Delta _{g,D}^{\left( p,q\right) }\left( f\right) \leq \left[ 
\frac{\overline{\tau }_{f,D}\left( m,q\right) }{\tau _{g,D}\left( m,p\right) 
}\right] ^{\frac{1}{\lambda _{g}\left( m,p\right) }}~.
\end{multline*}
\end{theorem}

\qquad In the next theorems we intend to find out $\left( p,q\right) $-th
relative Gol'dberg type ( resp. $\left( p,q\right) $-th relative Gol'dberg
lower type, $\left( p,q\right) $-th relative Gol'dberg weak type) of an
entire function $f\left( z\right) $\ with respect to another entire function 
$g\left( z\right) $\ (both $f\left( z\right) $ and $g\left( z\right) $ are
of $n$- complex variables ) when $\left( m,q\right) $-th relative Gol'dberg
type (resp. $\left( m,q\right) $-th relative Gol'dberg lower type, $\left(
m,q\right) $-th relative Gol'dberg weak type) of $f\left( z\right) $\ and $%
\left( m,p\right) $-th relative Gol'dberg type (resp. $\left( m,p\right) $%
-th relative Gol'dberg lower type, $\left( m,p\right) $-th relative
Gol'dberg weak type) of $g\left( z\right) $\ with respect to another entire
function $h\left( z\right) $\ ($h\left( z\right) $ is also of $n$- complex
variables ) are given where $m\geq p\geq 1$ and $m\geq q\geq 1.$Basically we
state the theorems without their proofs as those can easily be carried out
after applying the same technique our previous discussion and with the help
of {\normalsize Theorem \ref{l2} and Remark \ref{r2}.}

\begin{theorem}
\label{t5.1} Let $f\left( z\right) $, $g\left( z\right) $ and $h\left(
z\right) $ be any three entire functions of $n$- complex variables and $D$
be a bounded complete $n$-circular domain with center at origin in $%
\mathbb{C}
^{n}.$ Also let $f\left( z\right) $ and $g\left( z\right) $ be entire
functions with relative index-pairs $\left( m,q\right) $ and $\left(
m,p\right) $ with respect to $h\left( z\right) $ respectively where $p,q,m$
are all positive integers$.$ If $\lambda _{h}^{\left( m,p\right) }\left(
g\right) =\rho _{h}^{\left( m,p\right) }\left( g\right) ,$ then%
\begin{multline*}
\left[ \frac{\overline{\sigma }_{h,D}^{\left( m,q\right) }\left( f\right) }{%
\sigma _{h,D}^{\left( m,p\right) }\left( g\right) }\right] ^{\frac{1}{\rho
_{h}^{\left( m,p\right) }\left( g\right) }}\leq \overline{\sigma }%
_{g,D}^{\left( p,q\right) }\left( f\right) \leq \min \left\{ \left[ \frac{%
\overline{\sigma }_{h,D}^{\left( m,q\right) }\left( f\right) }{\overline{%
\sigma }_{h,D}^{\left( m,p\right) }\left( g\right) }\right] ^{\frac{1}{\rho
_{h}^{\left( m,p\right) }\left( g\right) }},\left[ \frac{\sigma
_{h,D}^{\left( m,q\right) }\left( f\right) }{\sigma _{h,D}^{\left(
m,p\right) }\left( g\right) }\right] ^{\frac{1}{\rho _{h}^{\left( m,p\right)
}\left( g\right) }}\right\} \\
\leq \max \left\{ \left[ \frac{\overline{\sigma }_{h,D}^{\left( m,q\right)
}\left( f\right) }{\overline{\sigma }_{h,D}^{\left( m,p\right) }\left(
g\right) }\right] ^{\frac{1}{\rho _{h}^{\left( m,p\right) }\left( g\right) }%
},\left[ \frac{\sigma _{h,D}^{\left( m,q\right) }\left( f\right) }{\sigma
_{h,D}^{\left( m,p\right) }\left( g\right) }\right] ^{\frac{1}{\rho
_{h}^{\left( m,p\right) }\left( g\right) }}\right\} \leq \sigma
_{g,D}^{\left( p,q\right) }\left( f\right) \leq \left[ \frac{\sigma
_{h,D}^{\left( m,q\right) }\left( f\right) }{\overline{\sigma }%
_{h,D}^{\left( m,p\right) }\left( g\right) }\right] ^{\frac{1}{\rho
_{h}^{\left( m,p\right) }\left( g\right) }}
\end{multline*}%
and%
\begin{multline*}
\left[ \frac{\tau _{h,,D}^{\left( m,q\right) }\left( f\right) }{\overline{%
\tau }_{h,D}^{\left( m,p\right) }\left( g\right) }\right] ^{\frac{1}{\lambda
_{h}^{\left( m,p\right) }\left( g\right) }}\leq \tau _{g,D}^{\left(
p,q\right) }\left( f\right) \leq \min \left\{ \left[ \frac{\tau
_{h,D}^{\left( m,q\right) }\left( f\right) }{\tau _{h,D}^{\left( m,p\right)
}\left( g\right) }\right] ^{\frac{1}{\lambda _{h}^{\left( m,p\right) }\left(
g\right) }},\left[ \frac{\overline{\tau }_{h,D}^{\left( m,q\right) }\left(
f\right) }{\overline{\tau }_{h,D}^{\left( m,p\right) }\left( g\right) }%
\right] ^{\frac{1}{\lambda _{h}^{\left( m,p\right) }\left( g\right) }%
}\right\} \\
\leq \max \left\{ \left[ \frac{\tau _{h,D}^{\left( m,q\right) }\left(
f\right) }{\tau _{h,D}^{\left( m,p\right) }\left( g\right) }\right] ^{\frac{1%
}{\lambda _{h}^{\left( m,p\right) }\left( g\right) }},\left[ \frac{\overline{%
\tau }_{h,D}^{\left( m,q\right) }\left( f\right) }{\overline{\tau }%
_{h,D}^{\left( m,p\right) }\left( g\right) }\right] ^{\frac{1}{\lambda
_{h}^{\left( m,p\right) }\left( g\right) }}\right\} \leq \overline{\tau }%
_{g,D}^{\left( p,q\right) }\left( f\right) \leq \left[ \frac{\overline{\tau }%
_{h,D}^{\left( m,q\right) }\left( f\right) }{\tau _{h,D}^{\left( m,p\right)
}\left( g\right) }\right] ^{\frac{1}{\lambda _{h}^{\left( m,p\right) }\left(
g\right) }}~.
\end{multline*}
\end{theorem}

\begin{theorem}
\label{t5.2} Let $f\left( z\right) $, $g\left( z\right) $ and $h\left(
z\right) $ be any three entire functions of $n$- complex variables and $D$
be a bounded complete $n$-circular domain with center at origin in $%
\mathbb{C}
^{n}.$ Also let $f\left( z\right) $ and $g\left( z\right) $ be entire
functions with relative index-pairs $\left( m,q\right) $ and $\left(
m,p\right) $ with respect to $h\left( z\right) $ respectively where $p,q,m$
are all positive integers$.$ If $f\left( z\right) $ is of regular relative $%
\left( m,q\right) $- Gol'dberg growth with respect to entire function $%
h\left( z\right) ,$ then%
\begin{multline*}
\left[ \frac{\overline{\sigma }_{h,D}^{\left( m,q\right) }\left( f\right) }{%
\sigma _{h,D}^{\left( m,p\right) }\left( g\right) }\right] ^{\frac{1}{\rho
_{h}^{\left( m,p\right) }\left( g\right) }}\leq \tau _{g,D}^{\left(
p,q\right) }\left( f\right) \leq \min \left\{ \left[ \frac{\overline{\sigma }%
_{h,D}^{\left( m,q\right) }\left( f\right) }{\overline{\sigma }%
_{h,D}^{\left( m,p\right) }\left( g\right) }\right] ^{\frac{1}{\rho
_{h}^{\left( m,p\right) }\left( g\right) }},\left[ \frac{\sigma
_{h,D}^{\left( m,q\right) }\left( f\right) }{\sigma _{h,D}^{\left(
m,p\right) }\left( g\right) }\right] ^{\frac{1}{\rho _{h}^{\left( m,p\right)
}\left( g\right) }}\right\} \\
\leq \max \left\{ \left[ \frac{\overline{\sigma }_{h,D}^{\left( m,q\right)
}\left( f\right) }{\overline{\sigma }_{h,D}^{\left( m,p\right) }\left(
g\right) }\right] ^{\frac{1}{\rho _{h}^{\left( m,p\right) }\left( g\right) }%
},\left[ \frac{\sigma _{h,D}^{\left( m,q\right) }\left( f\right) }{\sigma
_{h,D,}^{\left( m,p\right) }\left( g\right) }\right] ^{\frac{1}{\rho
_{h}^{\left( m,p\right) }\left( g\right) }}\right\} \leq \overline{\tau }%
_{g,D}^{\left( p,q\right) }\left( f\right) \leq \left[ \frac{\sigma
_{h,D}^{\left( m,q\right) }\left( f\right) }{\overline{\sigma }%
_{h,D}^{\left( m,p\right) }\left( g\right) }\right] ^{\frac{1}{\rho
_{h}^{\left( m,p\right) }\left( g\right) }}
\end{multline*}%
and%
\begin{multline*}
\left[ \frac{\tau _{h,D}^{\left( m,q\right) }\left( f\right) }{\overline{%
\tau }_{h,D}^{\left( m,p\right) }\left( g\right) }\right] ^{\frac{1}{\lambda
_{h}^{\left( m,p\right) }\left( g\right) }}\leq \overline{\sigma }%
_{g,D}^{\left( p,q\right) }\left( f\right) \leq \min \left\{ \left[ \frac{%
\tau _{h,D}^{\left( m,q\right) }\left( f\right) }{\tau _{h,D}^{\left(
m,p\right) }\left( g\right) }\right] ^{\frac{1}{\lambda _{h}^{\left(
m,p\right) }\left( g\right) }},\left[ \frac{\overline{\tau }_{h,D}^{\left(
m,q\right) }\left( f\right) }{\overline{\tau }_{h,D}^{\left( m,p\right)
}\left( g\right) }\right] ^{\frac{1}{\lambda _{h}^{\left( m,p\right) }\left(
g\right) }}\right\} \\
\leq \max \left\{ \left[ \frac{\tau _{h,D}^{\left( m,q\right) }\left(
f\right) }{\tau _{h,D}^{\left( m,p\right) }\left( g\right) }\right] ^{\frac{1%
}{\lambda _{h}^{\left( m,p\right) }\left( g\right) }},\left[ \frac{\overline{%
\tau }_{h,D}^{\left( m,q\right) }\left( f\right) }{\overline{\tau }%
_{h,D}^{\left( m,p\right) }\left( g\right) }\right] ^{\frac{1}{\lambda
_{h}^{\left( m,p\right) }\left( g\right) }}\right\} \leq \sigma
_{g,D}^{\left( p,q\right) }\left( f\right) \leq \left[ \frac{\overline{\tau }%
_{h,D}^{\left( m,q\right) }\left( f\right) }{\tau _{h,D}^{\left( m,p\right)
}\left( g\right) }\right] ^{\frac{1}{\lambda _{h}^{\left( m,p\right) }\left(
g\right) }}~.
\end{multline*}
\end{theorem}

\begin{theorem}
\label{t5.5} Let $f\left( z\right) $, $g\left( z\right) $ and $h\left(
z\right) $ be any three entire functions of $n$- complex variables and $D$
be a bounded complete $n$-circular domain with center at origin in $%
\mathbb{C}
^{n}.$ Also let $f\left( z\right) $ and $g\left( z\right) $ be entire
functions with relative index-pairs $\left( m,q\right) $ and $\left(
m,p\right) $ with respect to $h\left( z\right) $ respectively where $p,q,m$
are all positive integers$.$ Then%
\begin{multline*}
\max \left\{ \left[ \frac{\overline{\sigma }_{h,D}^{\left( m,q\right)
}\left( f\right) }{\tau _{h,D}^{\left( m,p\right) }\left( g\right) }\right]
^{\frac{1}{\lambda _{h}^{\left( m,p\right) }\left( g\right) }},\left[ \frac{%
\sigma _{h,D}^{\left( m,q\right) }\left( f\right) }{\overline{\tau }%
_{h,D}^{\left( m,p\right) }\left( g\right) }\right] ^{\frac{1}{\lambda
_{h}^{\left( m,p\right) }\left( g\right) }}\right\} \leq \sigma
_{g,D}^{\left( p,q\right) }\left( f\right) \\
\leq \min \left\{ \left[ \frac{\overline{\tau }_{h,D}^{\left( m,q\right)
}\left( f\right) }{\tau _{h,D}^{\left( m,p\right) }\left( g\right) }\right]
^{\frac{1}{\lambda _{h}^{\left( m,p\right) }\left( g\right) }},\left[ \frac{%
\sigma _{h,D}^{\left( m,q\right) }\left( f\right) }{\overline{\sigma }%
_{h,D}^{\left( m,p\right) }\left( g\right) }\right] ^{\frac{1}{\rho
_{h}^{\left( m,p\right) }\left( g\right) }},\left[ \frac{\overline{\tau }%
_{h,D}^{\left( m,q\right) }\left( f\right) }{\overline{\sigma }%
_{h,D}^{\left( m,p\right) }\left( g\right) }\right] ^{\frac{1}{\rho
_{h}^{\left( m,p\right) }\left( g\right) }}\right\}
\end{multline*}%
and%
\begin{multline*}
\left[ \frac{\overline{\sigma }_{h,D}^{\left( m,q\right) }\left( f\right) }{%
\overline{\tau }_{h,D}^{\left( m,p\right) }\left( g\right) }\right] ^{\frac{1%
}{\lambda _{h}^{\left( m,p\right) }\left( g\right) }}\leq \overline{\sigma }%
_{g,D}^{\left( p,q\right) }\left( f\right) \\
\leq \min \left\{ 
\begin{array}{c}
\left[ \frac{\overline{\sigma }_{h,D}^{\left( m,q\right) }\left( f\right) }{%
\overline{\sigma }_{h,D}^{\left( m,p\right) }\left( g\right) }\right] ^{%
\frac{1}{\rho _{h}^{\left( m,p\right) }\left( g\right) }},\left[ \frac{%
\sigma _{h,D}^{\left( m,q\right) }\left( f\right) }{\sigma _{h,D}^{\left(
m,p\right) }\left( g\right) }\right] ^{\frac{1}{\rho _{h}^{\left( m,p\right)
}\left( g\right) }},\left[ \frac{\tau _{h,D}^{\left( m,q\right) }\left(
f\right) }{\tau _{h,D}^{\left( m,p\right) }\left( g\right) }\right] ^{\frac{1%
}{\lambda _{h}^{\left( m,p\right) }\left( g\right) }}, \\ 
\left[ \frac{\overline{\tau }_{h,D}^{\left( m,q\right) }\left( f\right) }{%
\overline{\tau }_{h,D}^{\left( m,p\right) }\left( g\right) }\right] ^{\frac{1%
}{\lambda _{h}^{\left( m,p\right) }\left( g\right) }},\left[ \frac{\overline{%
\tau }_{h,D}^{\left( m,q\right) }\left( f\right) }{\sigma _{h,D}^{\left(
m,p\right) }\left( g\right) }\right] ^{\frac{1}{\rho _{h}^{\left( m,p\right)
}\left( g\right) }},\left[ \frac{\tau _{h,D}^{\left( m,q\right) }\left(
f\right) }{\overline{\sigma }_{h,D}^{\left( m,p\right) }\left( g\right) }%
\right] ^{\frac{1}{\rho _{h}^{\left( m,p\right) }\left( g\right) }}%
\end{array}%
\right\} ~.
\end{multline*}
\end{theorem}

\begin{theorem}
\label{t5.6} Let $f\left( z\right) $, $g\left( z\right) $ and $h\left(
z\right) $ be any three entire functions of $n$- complex variables and $D$
be a bounded complete $n$-circular domain with center at origin in $%
\mathbb{C}
^{n}.$ Also let $f\left( z\right) $ and $g\left( z\right) $ be entire
functions with relative index-pairs $\left( m,q\right) $ and $\left(
m,p\right) $ with respect to $h\left( z\right) $ respectively where $p,q,m$
are all positive integers$.$ Then%
\begin{multline*}
\max \left\{ 
\begin{array}{c}
\left[ \frac{\overline{\tau }_{h,D}^{\left( m,q\right) }\left( f\right) }{%
\overline{\tau }_{h,D}^{\left( m,p\right) }\left( g\right) }\right] ^{\frac{1%
}{\lambda _{h}^{\left( m,p\right) }\left( g\right) }},\left[ \frac{\tau
_{h,D}^{\left( m,q\right) }\left( f\right) }{\tau _{h,D}^{\left( m,p\right)
}\left( g\right) }\right] ^{\frac{1}{\lambda _{h}^{\left( m,p\right) }\left(
g\right) }},\left[ \frac{\overline{\sigma }_{h,D}^{\left( m,q\right) }\left(
f\right) }{\overline{\sigma }_{h,D}^{\left( m,p\right) }\left( g\right) }%
\right] ^{\frac{1}{\rho _{h}^{\left( m,p\right) }\left( g\right) }}, \\ 
\left[ \frac{\sigma _{h,D}^{\left( m,q\right) }\left( f\right) }{\sigma
_{h,D}^{\left( m,p\right) }\left( g\right) }\right] ^{\frac{1}{\rho
_{h}^{\left( m,p\right) }\left( g\right) }},\left[ \frac{\sigma
_{h,D}^{\left( m,q\right) }\left( f\right) }{\overline{\tau }_{h,D}^{\left(
m,p\right) }\left( g\right) }\right] ^{\frac{1}{\lambda _{h}^{\left(
m,p\right) }\left( g\right) }},\left[ \frac{\overline{\sigma }_{h,D}^{\left(
m,q\right) }\left( f\right) }{\tau _{h,D}^{\left( m,p\right) }\left(
g\right) }\right] ^{\frac{1}{\lambda _{h}^{\left( m,p\right) }\left(
g\right) }}%
\end{array}%
\right\} \\
\leq \overline{\tau }_{g,D}^{\left( p,q\right) }\left( f\right) \leq \left[ 
\frac{\tau _{h,D}^{\left( m,q\right) }\left( f\right) }{\overline{\sigma }%
_{h,D}^{\left( m,p\right) }\left( g\right) }\right] ^{\frac{1}{\rho
_{h}^{\left( m,p\right) }\left( g\right) }}
\end{multline*}%
and%
\begin{multline*}
\max \left\{ \left[ \frac{\overline{\sigma }_{h,D}^{\left( m,q\right)
}\left( f\right) }{\sigma _{h,D}^{\left( m,p\right) }\left( g\right) }\right]
^{\frac{1}{\rho _{h}^{\left( m,p\right) }\left( g\right) }},\left[ \frac{%
\tau _{h,D}^{\left( m,q\right) }\left( f\right) }{\overline{\tau }%
_{h,D}^{\left( m,p\right) }\left( g\right) }\right] ^{\frac{1}{\lambda
_{h}^{\left( m,p\right) }\left( g\right) }},\left[ \frac{\overline{\sigma }%
_{h,D}^{\left( m,q\right) }\left( f\right) }{\overline{\tau }_{h,D}^{\left(
m,p\right) }\left( g\right) }\right] ^{\frac{1}{\lambda _{h}^{\left(
m,p\right) }\left( g\right) }}\right\} \leq \tau _{g,D}^{\left( p,q\right)
}\left( f\right) \\
\leq \min \left\{ \left[ \frac{\tau _{h,D}^{\left( m,q\right) }\left(
f\right) }{\overline{\sigma }_{h,D}^{\left( m,p\right) }\left( g\right) }%
\right] ^{\frac{1}{\rho _{h}^{\left( m,p\right) }\left( g\right) }},\left[ 
\frac{\overline{\tau }_{h,D}^{\left( m,q\right) }\left( f\right) }{\sigma
_{h,D}^{\left( m,p\right) }\left( g\right) }\right] ^{\frac{1}{\rho
_{h}^{\left( m,p\right) }\left( g\right) }}\right\} ~.
\end{multline*}
\end{theorem}

\end{document}